\documentclass{elsarticle}

 \usepackage{graphicx}
\usepackage{color}
 \usepackage{epsfig}
\usepackage{subfig}
\usepackage{csquotes}
\usepackage{amssymb}
\usepackage{amsmath}
\usepackage{hyperref}

\newtheorem{theorem}{Theorem}[section]
\newtheorem{lemma}[theorem]{Lemma}

\newenvironment{proof}[1][Proof]{\begin{trivlist}
\item[\hskip \labelsep {\bfseries #1}]}{\end{trivlist}}

\newenvironment{remark}[1][Remark]{\begin{trivlist}
\item[\hskip \labelsep {\bfseries #1}]}{\end{trivlist}}

\begin{document}

\begin{frontmatter}

\title{ A compact high order Alternating Direction Implicit method for three-dimensional  acoustic wave equation with variable coefficient}

\author[add1]{Keran Li}
\ead{keran.li1@ucalgary.ca}
\author[add1]{Wenyuan Liao\corref{cor1}}
\ead{wliao@ucalgary.ca}
\author[add1]{Yaoting Lin}
\ead{yaoting.lin@ucalgary.ca}
\cortext[cor1]{Corresponding author}

\address[add1]{Department of Mathematics \& Statistics, University of Calgary, AB, T2N 1N4, Canada}

\begin{abstract}
Efficient and accurate numerical simulation of seismic wave propagation is important in various Geophysical applications such as 
seismic full waveform  inversion (FWI) problem. However,
due to the large size of the physical domain and requirement on low numerical dispersion, many existing
numerical methods are inefficient  for numerical modelling of seismic wave propagation in an heterogenous media.  
Despite the great efforts that have been devoted during the past decades, it still remains a challenging task
 in the development of efficient and accurate finite difference method
for multi-dimensional acoustic wave equation  with variable velocity.
In this paper we proposed a Pad\'{e}  approximation based finite difference scheme for solving 
the acoustic wave equation in three-dimensional heterogeneous media.
The new method is obtained by combining the Pad\'{e} approximation and a novel algebraic manipulation. 
The efficiency of the new algorithm is further improved through the Alternative Directional Implicit (ADI) method. 
The stability of the new algorithm has been theoretically proved by energy method. The new method is conditionally stable with a better 
 Courant - Friedrichs - Lewy condition (CFL)  condition, which has been verified numerically.  Extensive numerical examples have been solved, which 
demonstrated that   the new method is accurate, efficient and stable.
\end{abstract}

\end{frontmatter}

\section{Introduction}
\label{sec:1}
Finite difference (FD) method  has been widely used 
in various science and engineering applications for several reasons such as easy implementation,  high efficiency etc, when
 the analytical solution is not available. 
One example is the acoustic wave equation when a non-zero point source function is included in the equation. 
 In particular, the high-order FD methods
have  attracted the interests of many researchers working on  
seismic modelling (see \cite{Chen2007, Cohen1996, Liu2009a,Takeuchi2000} and references therein)  due to
the high-order accuracy and  effectiveness in suppressing numerical dispersion. Moreover, the number of grid points per wavelength
required by higher order FD methods is significantly less than that of the conventional FD methods.

Recently, a great deal of efforts have been devoted to develop high-order FD schemes 
for the  acoustic equations, and many accurate and
efficient methods have been developed. Levander \cite{Levander1988} addressed the cost-effectiveness of solving real 
problems using high-order spatial derivatives to allow a more coarse spatial sample rate.
In \cite{Liu2009a}, the authors used a plane wave theory and the Taylor series expansion to develop a low dispersion 
time-space domain FD scheme with error O$({\tau}^2+h^{2M})$ for 1-D, 2-D and 3-D acoustic wave equations,
where $\tau$ and $h$ represent the time step and spatial grid size, respectively.
It was then shown  that, along certain fixed directions the error can be improved to $O({\tau}^{2M} + h^{2M})$.
In \cite{Cohen1996}, Cohen and Poly extended the works of Dablain \cite{Dablain1986}, Shubin and Bell \cite{Shubin1987} and Bayliss et al. \cite{Bayliss1986}
and developed a fourth-order accurate explicit scheme with error of O$({\tau}^4+h^{4})$ to solve the heterogeneous acoustic
wave equation.
Moreover, it has been reported that highly accurate numerical methods are very effective in suppressing the annoying numerical dispersion \cite{Finkelstein2007}. 
High-order FD method is of particular importance for large-scale 3D acoustic wave equation, as it requires less grid points\cite{Etgen2007}.

These  methods are accurate but are non-compact, which give rise to two issues: efficiency and
difficulty in boundary condition treatment. 
For example, the conventional non-compact fourth-order FD scheme requires
a five-point stencil in 1D to approximate $u_{xx}$,  while in 2D, a  9-point stencil is needed in  approximating  $\Delta u$.  In 3D problem,  
a 13-point stencil is required to approximate the derivative $u_{xx} + u_{yy} + u_{zz}$ with fourth-order accuracy.
To overcome these difficulties,  a variety of compact higher-order FD
schemes  have been developed for hyperbolic, parabolic  and elliptical 
partial differential equations. In \cite{Chu1998}, the authors developed a
 family of fourth-order three-point combined difference schemes to 
approximate the first- and second-order spatial derivatives.
  In \cite{Kim2007}, the authors introduced a family of three-level implicit FD schemes 
which incorporate the locally one-dimensional method. 
For more recent compact higher-order difference methods, the readers are referred to
\cite{Shukla2005}.

 For three-dimensional problems, an implicit scheme results in 
 a block tridiagonal system, which is  solved at each time step. Direct solution of such large block  linear system is very inefficient, 
therefore, some operator splitting techniques are used to convert 
the three-dimensional problem  into a sequence of one-dimensional problems.  One widely used method is the ADI method, which was originally introduced by Peaceman and Rachford \cite{Peaceman1955}
to solve parabolic and elliptic equations. 
Since then a lot of developments have been made over the years for  
hyperbolic equations \cite{Douglas1966,Lees1962}. 
Later, Fairweather and Mitchell  \cite{Fairweather1965} developed a fourth-order compact ADI scheme  for solving the wave equation. 
Some other related work can be found in \cite{Li1991, Ristow1997}. 
Combined with Pad\'{e} approximation of the finite difference operator,
some efficient and high-order compact finite difference methods have been developed to solve the acoustic wave equations
in 2D and 3D with constant velocity \cite{Das2014, Liao2014}. However when the velocity is a spatially varying function, it is difficult to apply this technique because the algebraic manipulation  is not applicable here. 
Nevertheless, some  research work on accurate and low-dispersion
numerical simulation of  acoustic wave propagation in heterogeneous media have been reported\cite{ Zhang2011}, in which 
 a layered model consisting of multiple horizontal homogeneous layers was considered, with 
the method development and stability analysis  are based on constant velocity model.
 
In \cite{Liao2018} a new fourth-order compact ADI FD scheme was proposed to solve  the two-dimensional acoustic wave equation
with spatially variable velocity. 
Here we  extended this method and the Pad\'{e} approximation based high-order compact FD 
scheme in \cite{Das2014}  to the 3D acoustic wave equation with non-constant velocity.  
The new method is compact and efficient, with  fourth-order  accuracy in both time and space. 
 To our best knowledge, it is the first time that the energy method has been used to conduct theoretical stability analysis of 
the Pad\'{e}-approximation based compact scheme. It has been shown that the obtained stability condition (CFL condition)
is highly consistent to the CFL condition obtained by other methods.
The rest of the paper is organized as the follows.
We first give a  brief introduction of the acoustic wave equation and several existing standard second-order central FD
 schemes, compact high-order FD schemes for spatial 
derivatives and some other related high-order method in Section 2, then
derive the new compact fourth-order ADI FD scheme in Section  3. Stability analysis of the new method is presented in Section 4, 
 which is followed by three numerical  examples 
in Section 5.  Finally, the conclusions and some future works are discussed in Section 6.

\section{Acoustic wave equation and existing algorithms}
Consider  the  3D acoustic wave equation  
\begin{eqnarray}
& &  \hspace{-0.2in} u_{tt}   =  {\nu}^2(x,y,z) (u_{xx} + u_{yy} + u_{zz} ) + s(x,y,z,t),  (x,y,z,t)\,\in\, \Omega \,\times \, [0,T],\label{exact}\\
& & \hspace{-0.2in}   u(x,y,z,0)  =  f_1(x,y,z), \quad  (x,y,z)\,\in \Omega,\label{init}\\
& &  \hspace{-0.2in}  u_t(x,y,z,0) =  f_2(x,y,z), \quad (x,y,z)\,\in \Omega \label{init-2}\\
& &  \hspace{-0.2in}  u(x,y,z,t)   =    g(x,y,z,t),\quad (x,y,z,t)\,\in\, \partial \Omega \,\times \, [0,T], \label{bdry-exact}
\end{eqnarray}
where  $\nu(x,y,z)$ represents  the wave velocity.  Here $\Omega \subset R^3$ is a finite computational domain and
$s(x,y,z,t)$ is the source function.  We denote  $c(x,y,z)  = {\nu}^2(x,y,z)$ for the sake of simple notation.

First  assume  that $\Omega$ is a 3D rectangular  box : $[x_0,  \ x_1] \times [y_0,  \  y_1 ] \times [z_0, \  z_1]$, which is discretized into an
 $N_x \times N_y \times N_z$ grid with spatial grid sizes  $
h_x  $,  $ h_y $ and $h_z$.
Let  $\tau $ be  the time stepsize and 
$u_{i,j,k}^n$ denote the numerical solution at the grid point $(x_i,y_j,z_k)$ and time level $n \tau$.
Let's  first define  the standard central difference operators 
\[
{\delta}_{t}^2   u_{i,j,k}^{n}=  u_{i,j,k}^{n-1}-2u_{i,j,k}^{n}+u_{i,j,k}^{n+1},
\]
\[
{\delta}_{x}^2 u_{i,j,k}^{n}   = u_{i-1,j,k}^n -2u_{i,j,k}^{n}+u_{i+1,j,k}^{n},
\]
\[
{\delta}_{y}^2 u_{i,j,k}^{n}   = u_{i,j-1,k}^n -2u_{i,j,k}^{n}+u_{i,j+1,k}^{n},
\]
\[
 {\delta}_{z}^2 u_{i,j,k}^{n}   = u_{i,j,k-1}^n -2u_{i,j,k}^{n}+u_{i,j,k+1}^{n}.
\]
The standard second-order central FD schemes are then given by
 approximating the  second derivatives in Eq.(\ref{exact}). For example, 
\begin{eqnarray}
u_{tt}(x_{i},y_{j}, z_k, t_{n})  & = & \frac{1}{{\tau}^2} \ {\delta}_t^2   u_{i,j,k}^{n}  + O({\tau}^2), \label{dt2}\\
u_{xx}(x_{i},y_{j},z_k,t_{n}) & =  & \frac{1}{h_x^2 } \ {\delta}_x^2 u_{i,j,k}^{n}  +O(h_x^2). \label{dx2}
\end{eqnarray}
The approximations of  $u_{yy}$ and $u_{zz}$ are similar to Eq. (\ref{dx2}).

To improve the method to fourth-order in space,
the conventional high-order FD method was derived by approximating the  derivatives using more than three points 
in one direction, which results in larger stencil. 
For instance, if $2M+1$ points are used to approximate $u_{xx}$, one can obtain the following formula
\begin{equation}\label{convention-1}
u_{xx} (x_i,y_j, z_k,t_n) \approx [a_0 u_{i,j,k}^n + \sum_{m=1}^M a_m (u_{i-m,j,k}^n + u_{i+m,j,k}^n)]/h_x^2,
\end{equation}
which can be as accurate as $(2M)th$-order  in $x$. 
The  conventional high-order
FD method is accurate in space  but suffers severe numerical dispersion. Another issue is that it requires more 
computer memory due to the 
larger stencil for implicit method. Moreover, more 
points are needed to approximate  the boundary condition.

To improve the accuracy in time, a  class of  time-domain high-order FD methods have been derived by Liu and Sen \cite{Liu2009a}. The idea of the time-domain high-order
FD method is to determine coefficients using time-space domain dispersion.  As a result, the coefficient will be a function of $\frac{v \tau}{h}$.
It was noted that in 1D case,
the time-domain high-order FD method can be as accurate as $(2M)th$-order  in both time and space, provided some conditions  are satisfied, while for multidimensional case,
$(2M)th$-order is also possible along some propagation directions.

To develop a high-order compact ADI FD scheme, we apply  the Pad\'{e} approximation to the second-order 
central FD operators, so the 
second  derivatives $u_{tt}, u_{xx}, u_{yy}$ and  $u_{zz}$ can approximated with fourth-order accuracy.  For example, 
\begin{eqnarray}
\frac{{\delta}_t^2}{{\tau}^2\left(1+ \frac{1}{12}{\delta}_t^{2}\right)}u_{i,j,k}^{n} & = & u_{tt}(x_i,y_j,z_k,t_n)+{\rm O}({\tau}^4),\label{padeq4}\\
\frac{{\delta}_x^2}{h_x^2\left(1+ \frac{1}{12}\delta_{x}^{2}\right)}u_{i,j,k}^{n} & = &u_{xx}(x_i,y_j,z_k,t_n)+{\rm O}(h_x^4).\label{padeq1}
\end{eqnarray} 
The fourth-order approximations of $u_{yy}$ and $u_{zz}$ can be obtained similarly.

Let ${\lambda}_x = \frac{{\tau}^2}{h_x^2}$,  ${\lambda}_y =\frac{ {\tau}^2}{h_y^2}$, ${\lambda}_z = \frac{ {\tau}^2}{h_z^2}$.  Substituting the fourth-order Pad\'{e} approximations 
into Eq. (\ref{exact}) gives
\begin{equation}\label{hoc-adi-1}
 \frac{{\delta}_t^2}{ 1+\frac{1}{12}{\delta}_t^2} u_{i,j,k}^n = \left[ \frac{ {\lambda}_x c_{i,j,k} \ {\delta}_x^2}{1+\frac{1}{12}{\delta}_x^2} 
+\frac{{\lambda}_y c_{i,j,k} \ {\delta}_y^2}{1+\frac{1}{12}{\delta}_y^2}  + \frac{{\lambda}_z c_{i,j,k} \ {\delta}_z^2}{1+\frac{1}{12}{\delta}_z^2}   \right]  u_{i,j,k}^n + {\tau}^2 s_{i,j,k}^n.
\end{equation}
Truncation error analysis shows that  the algorithm is fourth-order accurate in time and space with the truncation error $ O({\tau}^4 + h_x^4 + h_y^4 + h_z^4 )$, provided the solution $u(x,y,z,t)$ and $c(x,y,z)$ satisfy certain smoothness conditions. 
As shown in \cite{Liao2018}, the difficulty to develop high-order compact scheme for wave equation with
 non-constant velocity is that, one can not multiply the operator 
\begin{equation}\label{operator-mul}
\left(1+\frac{{\delta}_t^2}{12}\right)\left(1+\frac{{\delta}_x^2}{12}\right)\left(1+\frac{{\delta}_y^2}{12}\right)\left(1+\frac{{\delta}_z^2}{12}\right)
\end{equation}
 to both sides of Eq. (\ref{hoc-adi-1}) to cancel the  fractional operators ${(1+\frac{{\delta}_t^2}{12})}^{-1}$,  ${(1+\frac{{\delta}_x^2}{12})}^{-1}$,
${(1+\frac{{\delta}_y^2}{12})}^{-1}$ and ${(1+\frac{{\delta}_z^2}{12})}^{-1}$, which are difficult to implement. To overcome this problem,
we develop an algebraic strategy which will be described in the next section.

\section{Derivation of the compact high-order ADI method}
Now  we extend the  novel algebraic manipulation introduced in \cite{Liao2018}  to the three-dimensional acoustic wave equation with 
variable velocity. 
We first demonstrate the difficulty in applying the Pad\'{e} approximation to the finite difference operator for solving  
the acoustic wave equation with non-constant velocity. 
Multiplying the operator in Eq. (\ref{operator-mul}) to Eq. (\ref{hoc-adi-1}) yields  
\begin{eqnarray}
 & & \hspace{-0.2in}  \left(1+\frac{1}{12}{\delta}_x^2\right)\left(1+\frac{1}{12}{\delta}_y^2\right) \left(1+\frac{1}{12}{\delta}_z^2\right)  {\delta}_t^2 u_{i,j,k}^n \nonumber \\
& &  \hspace{-0.2in} = {\lambda}_x  \left(1+\frac{1}{12}{\delta}_t^2\right)  \left(1+\frac{1}{12}{\delta}_x^2\right)\left(1+\frac{1}{12}{\delta}_y^2\right) \left(1+\frac{1}{12}{\delta}_z^2\right)  c_{i,j,k} \ \frac{{\delta}_x^2}{ 1+\frac{1}{12}{\delta}_x^2\ }u_{i,j,k}^n \nonumber \\
& &  \hspace{-0.2in} + {\lambda}_y  \left(1+\frac{1}{12}{\delta}_t^2\right)  \left(1+\frac{1}{12}{\delta}_x^2\right)\left(1+\frac{1}{12}{\delta}_y^2\right)  \left(1+\frac{1}{12}{\delta}_z^2\right) c_{i,j,k} \  \frac{{\delta}_y^2}{ 1+\frac{1}{12}{\delta}_y^2}u_{i,j,k}^n \nonumber \\
& &  \hspace{-0.2in} + {\lambda}_z  \left(1+\frac{1}{12}{\delta}_t^2\right)  \left(1+\frac{1}{12}{\delta}_x^2\right)\left(1+\frac{1}{12}{\delta}_y^2\right)  \left(1+\frac{1}{12}{\delta}_z^2\right) c_{i,j,k} \ \frac{{\delta}_z^2}{ 1+\frac{1}{12}{\delta}_z^2 }u_{i,j,k}^n \nonumber \\
& &  \hspace{-0.2in} + {\tau}^2 \left(1+\frac{1}{12}{\delta}_t^2\right)  \left(1+\frac{1}{12}{\delta}_x^2\right)\left(1+\frac{1}{12}{\delta}_y^2\right) \left(1+\frac{1}{12}{\delta}_z^2\right)s_{i,j,k}^n.
\label{operator-issue-1}
\end{eqnarray}
We use the first term on the right-hand side to illustrate the issue here.
Since $(1+ \frac{{\delta}_x^2}{12}) $, $(1+ \frac{{\delta}_y^2}{12}) $  and $(1+ \frac{{\delta}_z^2}{12}) $ are commutative, we change the order of the three finite difference operators
to obtain
\begin{eqnarray}
& & \hspace{-0.25in} {\lambda}_x \left(1+\frac{1}{12}{\delta}_t^2\right)  \left(1+\frac{1}{12}{\delta}_x^2\right)\left(1+\frac{1}{12}{\delta}_y^2\right)  \left(1+\frac{1}{12}{\delta}_z^2\right)c_{i,j,k} \ \frac{{\delta}_x^2}{ 1+\frac{1}{12}{\delta}_x^2 }u_{i,j,k}^n  =  \nonumber \\
& &  \hspace{-0.25in} {\lambda}_x  \left(1+\frac{1}{12}{\delta}_t^2\right)  \left(1+\frac{1}{12}{\delta}_y^2\right)\left(1+\frac{1}{12}{\delta}_z^2\right)  \left(1+\frac{1}{12}{\delta}_x^2\right) 
c_{i,j,k} \ \frac{{\delta}_x^2}{1+\frac{1}{12}{\delta}_x^2 }u_{i,j}^n. \label{operator-issue-2}
\end{eqnarray}


As discussed previously, when $c(x,y,z)$ is non-constant in terms of $x$, the operator $(1+\frac{1}{12}{\delta}_x^2)$ and $c_{i,j,k}$ are not commutative. Hence,
$\left(1+\frac{1}{12}{\delta}_x^2\right) \ c_{i,j,k} \neq  c_{i,j,k} \left(1+\frac{1}{12}{\delta}_x^2\right)$. Therefore,
  the operator 
$\left(1+\frac{1}{12}{\delta}_x^2\right)$ does not cancel the operator $\left(1+\frac{1}{12}{\delta}_x^2\right)^{-1}$. In other words,
\begin{eqnarray}
& & \hspace{-0.35in}{\lambda}_x \left(1+\frac{1}{12}{\delta}_t^2\right)\left(1+\frac{1}{12}{\delta}_y^2\right) \left(1+\frac{1}{12}{\delta}_z^2\right) \left(1+\frac{1}{12}{\delta}_x^2\right) 
c_{i,j,k}  \frac{ {\delta}_x^2}{1+\frac{1}{12}{\delta}_x^2}  u_{i,j,k}^n
  \nonumber \\
& &  \hspace{-0.25in} \neq  {\lambda}_x  \left(1+\frac{1}{12}{\delta}_t^2\right)\left(1+\frac{1}{12}{\delta}_y^2\right) \left(1+\frac{1}{12}{\delta}_z^2\right) 
c_{i,j,k} \   {\delta}_x^2 u_{i,j,k}^n.
\end{eqnarray}

To solve this problem, a novel factorization technique is used to preserve the compactness and fourth-order convergence of the numerical scheme.
 Multiplying $(1 + \frac{{\delta}_t^2}{12})$ to  Eq. (\ref{hoc-adi-1}) yields
\begin{eqnarray}\label{hoc-adi-new}
& & \hspace{-0.5in}{\delta}_t^2 u_{i,j,k}^n = c_{i,j,k}\left[  {\lambda}_x \left(1 + \frac{{\delta}_t^2}{12}\right)\frac{ {\delta}_x^2}{1+\frac{{\delta}_x^2}{12}} 
+ {\lambda}_y \left(1 + \frac{{\delta}_t^2}{12}\right) \frac{ {\delta}_y^2}{1+\frac{{\delta}_y^2}{12}} + \right. \nonumber \\
  & & \left.  {\lambda}_z \left(1 + \frac{{\delta}_t^2}{12}\right) \frac{ {\delta}_z^2}{1+\frac{{\delta}_z^2}{12}}  \right]  u_{i,j,k}^n + {\tau}^2 \left(1 + \frac{{\delta}_t^2}{12}\right)s_{i,j,k}^n.
\end{eqnarray}
Collecting the term ${\delta}_t^2 u_{i,j,k}^n$, we have
\begin{eqnarray}
& &\hspace{-0.25in} \left [ 1-\frac{{\lambda}_x c_{i,j,k}}{12}  \frac{{\delta}_x^2}{ 1+\frac{1}{12}{\delta}_x^2 } -\frac{{\lambda}_y c_{i,j,k}}{12}  \frac{{\delta}_y^2}{ 1+\frac{1}{12}{\delta}_y^2 } -\frac{{\lambda}_y c_{i,j,k}}{12}  \frac{{\delta}_z^2}{ 1+\frac{1}{12}{\delta}_z^2 }  \right] {\delta}_t^2 u_{i,j,k}^n  \nonumber \\
& & \hspace{-0.25in} = c_{i,j,k} \left[  \frac{ {\lambda}_x  {\delta}_x^2}{1+\frac{{\delta}_x^2}{12}} 
+ \frac{  {\lambda}_y {\delta}_y^2}{1+\frac{{\delta}_y^2}{12}}   + \frac{  {\lambda}_z {\delta}_z^2}{1+\frac{{\delta}_z^2}{12}}\right]  u_{i,j,k}^n + {\tau}^2 \left(1 + \frac{{\delta}_t^2}{12}\right)s_{i,j,k}^n.\label{hoc-adi-new-2}
\end{eqnarray}

Factoring the left-hand side of Eq. (\ref{hoc-adi-new-2}) yields
\begin{eqnarray}
& &\hspace{-0.25in} \left[ 1-\frac{ c_{i,j,k}}{12}  \frac{ {\lambda}_x {\delta}_x^2}{ 1+\frac{1}{12}{\delta}_x^2 }\right] \cdot \left[1 -\frac{c_{i,j,k}}{12}  \frac{{\lambda}_y  {\delta}_y^2}{ 1+\frac{1}{12}{\delta}_y^2 } \right] \cdot \left[1 -\frac{ c_{i,j,k}}{12}  \frac{{\lambda}_z{\delta}_z^2}{ 1+\frac{1}{12}{\delta}_z^2 }  \right] {\delta}_t^2 u_{i,j,k}^n = \nonumber \\
& &  \hspace{-0.38in} c_{i,j,k} \left[   \frac{ {\lambda}_x {\delta}_x^2}{1+\frac{1}{12}{\delta}_x^2} 
+ \frac{{\lambda}_y   {\delta}_y^2}{1+\frac{1}{12}{\delta}_y^2} + \frac{{\lambda}_z  {\delta}_z^2}{1+\frac{1}{12}{\delta}_z^2}  \right]  u_{i,j,k}^n + {\tau}^2 \left(1 + \frac{{\delta}_t^2}{12}\right)s_{i,j,k}^n + ERR,\label{hoc-adi-new-3}
\end{eqnarray}
where the factorization error is given by
\begin{eqnarray}
& & ERR =  \frac{{\lambda}_x {\lambda}_y }{144} \ c_{i,j,k}  \frac{{\delta}_x^2}{ 1+\frac{{\delta}_x^2}{12} } \    c_{i,j,k} 
 \frac{{\delta}_y^2}{ 1+\frac{{\delta}_y^2}{12} } {\delta}_t^2 u_{i,j,k}^n + \nonumber \\
& &  \frac{{\lambda}_y {\lambda}_z }{144} \   \frac{c_{i,j,k} \ {\delta}_y^2}{ 1+\frac{{\delta}_y^2}{12} } \    \frac{c_{i,j,k} \ {\delta}_z^2}{ 1+\frac{{\delta}_z^2}{12} } {\delta}_t^2 u_{i,j,k}^n +   \frac{{\lambda}_x {\lambda}_z }{144} \  \frac{c_{i,j,k}  \ {\delta}_x^2}{ 1+\frac{{\delta}_x^2}{12} } \   \frac{c_{i,j,k}  \ {\delta}_z^2}{ 1+\frac{{\delta}_z^2}{12} } {\delta}_t^2 u_{i,j,k}^n \nonumber \\
& & - \frac{{\lambda}_x {\lambda}_y {\lambda}_z   }{1728} \   \frac{c_{i,j,k}  \ {\delta}_x^2}{ 1+\frac{{\delta}_x^2}{12} } \   \frac{c_{i,j,k}  \ {\delta}_y^2}{ 1+\frac{{\delta}_y^2}{12} } \  \frac{c_{i,j,k}  
\ {\delta}_z^2}{ 1+\frac{{\delta}_z^2}{12} }  {\delta}_t^2 u_{i,j,k}^n  \label{error-term-1}.
\end{eqnarray}
Using Taylor series, one can verify that $ERR = O({\tau}^6)$, provided that $c(x,y,z)$ and $u(x,y,z,t)$ satisfy 
some conditions on smoothness. The  result regarding the order of the  truncation error is included in  the following theorem.
\begin{theorem}
Assume that  $u(x,y,z, t) \in C_{x,y,z,t}^{6,6,6,6}(\Omega \times [0,T])$ is the solution of the acoustic wave equation defined by
Eqs. ( \ref{exact} - \ref{bdry-exact}), and
the coefficient  satisfies the smoothness condition $c(x,y,z) \in C_{x,y,z}^{2,2,2}(\Omega)$.  Then  the  truncation error  given in Eq. (\ref{error-term-1})  satisfies
\[
 ERR = O({\tau}^6) + O(h_x^6) + O(h_y^6) + O(h_z^6),
\]
where $\tau, h_x, h_y$ and $h_z$ are the step sizes in  time, $x, y$ and $z$, respectively. 
\end{theorem}
\begin{proof}
A detailed proof of the theorem is given in Appendix A.
\qed
\end{proof}

\begin{remark}[Remark:] If  Eq. (\ref{hoc-adi-new-2}) is factorized in a different order of ${\delta}_x^2$, ${\delta}_y^2$ and ${\delta}_z^2$, for instance as
\begin{eqnarray}
& &\hspace{-0.45in} \left[1 -\frac{ c_{i,j,k}}{12}  \frac{{\lambda}_y{\delta}_y^2}{ 1+\frac{1}{12}{\delta}_y^2 } \right] \cdot  \left [ 1-\frac{ c_{i,j,k}}{12}  \frac{{\lambda}_x{\delta}_x^2}{ 1+\frac{1}{12}{\delta}_x^2 }\right] \cdot 
\left[1 -\frac{ c_{i,j,k}}{12}  \frac{{\lambda}_z{\delta}_z^2}{ 1+\frac{1}{12}{\delta}_z^2 } \right]{\delta}_t^2 u_{i,j,k}^n = \nonumber \\
& &  \hspace{-0.48in}  c_{i,j,k} \left[   \frac{ {\lambda}_x{\delta}_x^2}{1+\frac{1}{12}{\delta}_x^2} 
+  \frac{{\lambda}_y {\delta}_y^2}{1+\frac{1}{12}{\delta}_y^2}   + \frac{  {\lambda}_z{\delta}_z^2}{1+\frac{1}{12}{\delta}_z^2}\right]  u_{i,j,k}^n + {\tau}^2 \left(1 + \frac{{\delta}_t^2}{12}\right)s_{i,j,k}^n +ERR,\label{hoc-adi-new-3a}
\end{eqnarray}
then the factoring error  $ERR$ is given by
\begin{eqnarray}
ERR & = &  \frac{{\lambda}_y}{144} \  c_{i,j,k}  \frac{{\delta}_y^2}{ 1+\frac{1}{12}{\delta}_y^2 }  \ {\lambda}_x  \ c_{i,j,k}  \frac{{\delta}_x^2}{ 1+\frac{1}{12}{\delta}_x^2 } \   {\delta}_t^2 u_{i,j,k}^n\nonumber \\
& & +  \frac{{\lambda}_y}{144} \  c_{i,j,k}  \frac{{\delta}_y^2}{ 1+\frac{1}{12}{\delta}_y^2 }  \ {\lambda}_x  \ c_{i,j,k}  \frac{{\delta}_x^2}{ 1+\frac{1}{12}{\delta}_x^2 } \   {\delta}_t^2 u_{i,j,k}^n \nonumber \\
& & + \frac{{\lambda}_y}{144} \  c_{i,j,k}  \frac{{\delta}_y^2}{ 1+\frac{1}{12}{\delta}_y^2 }  \ {\lambda}_x  \ c_{i,j,k}  \frac{{\delta}_x^2}{ 1+\frac{1}{12}{\delta}_x^2 } \   {\delta}_t^2 u_{i,j,k}^n \nonumber \\
& & -\frac{1}{1728} {\lambda}_y \  c_{i,j,k}  \frac{{\delta}_y^2}{ 1+\frac{1}{12}{\delta}_y^2 }  \ {\lambda}_x  \ c_{i,j,k}  \frac{{\delta}_x^2}{ 1+\frac{1}{12}{\delta}_x^2 } \   {\delta}_t^2 u_{i,j,k}^n,\label{error-term-1-aa}
\end{eqnarray}
which has the same error estimation as that defined in Eq. (\ref{error-term-12}).
\end{remark}

Ignoring  the factoring error ERR in Eq. (\ref{hoc-adi-new-3}) leads to 
 the following compact  fourth-order FD  method
\begin{eqnarray}
& &\hspace{-0.15in} \left[1- \frac{{\lambda}_x c_{i,j,k}}{12} \ \frac{{\delta}_x^2}{1+\frac{{\delta}_x^2}{12}}\right] \cdot 
\left[1- \frac{{\lambda}_y c_{i,j,k}}{12} \ \frac{{\delta}_y^2}{1+\frac{{\delta}_y^2}{12}}\right] \cdot  \left[1- \frac{{\lambda}_z c_{i,j,k}}{12} \ \frac{{\delta}_z^2}{1+\frac{{\delta}_z^2}{12}}\right] {\delta}_t^2 u_{i,j,k}^n   \nonumber \\
& & \hspace{-0.15in}
=c_{i,j,k}  \left[ \frac{{\lambda}_x {\delta}_x^2}{1+\frac{{\delta}_x^2}{12}}  +    \frac{{\lambda}_y {\delta}_y^2}{1+\frac{{\delta}_y^2}{12}} +  \frac{ {\lambda}_z {\delta}_z^2}{1+\frac{{\delta}_z^2}{12}}\right] u_{i,j,k}^n + 
{\tau}^2 \left( 1+\frac{{\delta}_t^2}{12} \right) s_{i,j,k}^n.\label{hoc-adi-4-b}
\end{eqnarray}

Using ADI method, Eq. (\ref{hoc-adi-4-b}) can be efficiently solved  in three steps
\begin{eqnarray}
& &  \hspace{-0.15in}  \left(1- \frac{{\lambda}_x  c_{i,j,k} }{12} \ \frac{{\delta}_x^2}{1+\frac{{\delta}_x^2}{12}}\right) u_{i,j,k}^{**} = 
 \left[  \frac{c_{i,j,k} \ {\lambda}_x{\delta}_x^2}{1+\frac{{\delta}_x^2}{12}}  +  \frac{c_{i,j,k} \ {\lambda}_y {\delta}_y^2}{1+\frac{{\delta}_y^2}{12}} 
+  \frac{ c_{i,j,k} \ {\lambda}_z{\delta}_z^2}{1+\frac{{\delta}_z^2}{12}}  \right] u_{i,j,k}^n   \nonumber \\
& &  \hspace{0.15in}  + {\tau}^2 \left( 1+\frac{{\delta}_t^2}{12} \right) s_{i,j,k}^n, \  \ 2 \le j \le N_y-1, \ 2 \le k \le N_z -1, \label{hoc-adi-5-a} \\
& &  \hspace{-0.15in}    \left(1- \frac{{\lambda}_y c_{i,j,k}}{12} \ \frac{{\delta}_y^2}{1+\frac{{\delta}_y^2}{12}}\right)  u_{i,j,k}^{*}   =   u_{i,j,k}^{**},  \ 2 \le i \le N_x-1, \ 2 \le k \le N_z -1,\label{hoc-adi-5-bb} \\
& &   \hspace{-0.15in}   \left(1- \frac{{\lambda}_z c_{i,j,k}}{12} \ \frac{{\delta}_z^2}{1+\frac{{\delta}_z^2}{12}}\right) {\delta}_t^2 u_{i,j,k}^n   =   u_{i,j,k}^*, \ 2 \le i \le N_x-1, \ 2 \le j \le N_y -1.\label{hoc-adi-5-b}
\end{eqnarray}

Apparently the three equations are difficult to implement due to the three operators  $\left(1+\frac{{\delta}_x^2}{12}\right)^{-1}$,
$\left(1+\frac{{\delta}_y^2}{12}\right)^{-1}$ and $\left(1+\frac{{\delta}_z^2}{12}\right)^{-1}$.
To overcome this problem, 
we apply the following strategy. Firstly, 
divide both sides of   Eq. (\ref{hoc-adi-5-a}) by  $c_{i,j,k}$, then   multiply  $\left(1+\frac{{\delta}_x^2}{12}\right)$,
 we have
\begin{eqnarray}
& & \hspace{-0.3in} \left[\left( 1+ \frac{{\delta}_x^2}{12}\right)\frac{1}{c_{i,j,k}}- \frac{{\lambda}_x }{12} \ {\delta}_x^2\right] u_{i,j,k}^{**} =   {\tau}^2 \left( 1+ \frac{{\delta}_x^2}{12}\right) \left( 1+\frac{{\delta}_t^2}{12} \right) \frac{s_{i,j,k}^n}{c_{i,j,k}} \nonumber \\
& &  \hspace{-0.3in}  + \left[{\lambda}_x {\delta}_x^2 +  {\lambda}_y  \left( 1+ \frac{{\delta}_x^2}{12}\right)  \frac{{\delta}_y^2}{1+\frac{{\delta}_y^2}{12}} 
+   {\lambda}_z  \left( 1+ \frac{{\delta}_x^2}{12}\right)  \frac{{\delta}_z^2}{1+ \frac{{\delta}_z^2}{12}} \right]
 u_{i,j,k}^n.  \label{hoc-adi-6}
\end{eqnarray}
Eq. (\ref{hoc-adi-6}) is still hard to  implement because of the terms $\frac{{\delta}_y^2}{1+\frac{{\delta}_y^2}{12}}$ and $\frac{{\delta}_z^2}{1+\frac{{\delta}_z^2}{12}}$. 
Substituting  $\frac{{\delta}_y^2}{1+\frac{{\delta}_y^2}{12}} u_{i,j,k}^n$ with ${\delta}_y^2\left(1 - \frac{{\delta}_y^2}{12}\right) u_{i,j,k}^n$,  
$ \ \frac{{\delta}_z^2}{1+\frac{{\delta}_z^2}{12}} u_{i,j,k}^n$ with ${\delta}_z^2\left(1 - \frac{{\delta}_z^2}{12}\right) u_{i,j,k}^n$, respectively, we obtain
\begin{eqnarray}
& & \hspace{-0.35in} \left[\left( 1+ \frac{{\delta}_x^2}{12}\right)\frac{1}{c_{i,j,k}}- {\lambda}_x  \ \frac{{\delta}_x^2}{12}\right] u_{i,j,k}^{**} =    {\tau}^2 \left( 1+ \frac{{\delta}_x^2}{12}\right) \left( 1+\frac{{\delta}_t^2}{12}\right) \frac{s_{i,j,k}^n}{c_{i,j,k}} + \nonumber \\
& &  \hspace{-0.15in}   \left[{\lambda}_x {\delta}_x^2 +  {\lambda}_y  \left( 1+ \frac{{\delta}_x^2}{12}\right)  {\delta}_y^2\left(1 - \frac{{\delta}_y^2}{12}\right)  
+  {\lambda}_z \left( 1+ \frac{{\delta}_x^2}{12}\right)  {\delta}_z^2\left(1 - \frac{{\delta}_z^2}{12}\right) \right]
 u_{i,j,k}^n, \nonumber \\
 & &   \hspace{-0.15in} \text{for}  \ 2 \le j \le N_y-1, \ 2 \le k \le N_z -1. \label{hoc-adi-7}
\end{eqnarray}
Note the difference between Eq. (\ref{hoc-adi-6}) and Eq. (\ref{hoc-adi-7}) is $O(h_y^6+h_z^6)$, therefore, the  method is still fourth-order  in space.
 It is worth to mention  that the right-hand side of Eq. (\ref{hoc-adi-7}) includes larger stencil in both $y$ and $z$ directions. 
Furthermore, larger stencil 
needs values of $u_{i,j,k}^n$ outside the boundary when $j=2, N_y-1$ and $k=2, N_z-1$. To overcome this problem, we use one-sided approximation to 
approximate the values outside of the boundary. For example, $u_{i,0,k}^n$ is approximated by a linear combination of $u_{i,1,k}^n,\cdots, u_{i,4,k}^n$
with fourth-order accuracy. This boundary treatment is not complicate in terms of implementation, since it only involves the values at time level $n$, which is known.  Similarly, 
 dividing $c_{i,j,k}$ then multiplying $\left( 1+\frac{{\delta}_y^2}{12}\right)$ to both sides of Eq. (\ref{hoc-adi-5-bb}) lead to
\begin{eqnarray}
& &  \hspace{-0.3in} \left[ \left(1+\frac{{\delta}_y^2}{12}\right) \frac{1}{c_{i,j,k}} - {\lambda}_y \ \frac{{\delta}_y^2}{12} \right]  u_{i,j,k}^*   =   \left(1+\frac{{\delta}_y^2}{12}\right)  \frac{u_{i,j,k}^{**}}{c_{i,j,k}}\label{hoc-adi-8} \\
& &  \hspace{-0.25in} \text{for } i = 2,3,\cdots, N_x-1, \  \ k = 2,3,\cdots, N_z-1.  \nonumber
\end{eqnarray}

Finally, Eq. (\ref{hoc-adi-5-b}) can be transformed to the equivalent  linear system
\begin{eqnarray}
 & &  \hspace{-0.21in}  \left[ \left(1+\frac{{\delta}_z^2}{12}\right) \frac{1}{c_{i,j,k}} - {\lambda}_z \ \frac{{\delta}_z^2}{12} \right]  {\delta}_t^2 u_{i,j,k}^n   =   \left(1+\frac{{\delta}_z^2}{12}\right) \frac{u_{i,j,k}^{*}}{c_{i,j,k}},  \label{hoc-adi-8c}\\
& &  \hspace{-0.21in} \text{for } i = 2,3,\cdots, N_x-1, \  \ j = 2,3,\cdots, N_y-1. \nonumber
\end{eqnarray}
It is noted that Eq. (\ref{hoc-adi-8c}) is equivalent to a three-level FD scheme  
\begin{eqnarray}
& & \hspace{-0.27in}  \left[ \left(1+\frac{{\delta}_z^2}{12}\right) \frac{1}{c_{i,j,k}} - \frac{{\lambda}_z }{12} \ {\delta}_z^2 \right]  u_{i,j,k}^{n+1} = 
  \left[ \left(1+\frac{{\delta}_z^2}{12}\right) \frac{1}{c_{i,j,k}} - \frac{{\lambda}_z }{12} \ {\delta}_z^2 \right]  (2 u_{i,j,k}^{n}- u_{i,j,k}^{n-1} )  \nonumber \\ 
& & \hspace{0.2in} +  \left(1+\frac{{\delta}_z^2}{12}\right) \frac{u_{i,j,k}^*}{c_{i,j,k}} , \ \ i = 2,3,\cdots, N_x-1, \  j= 2,3,\cdots, N_y-1. \label{hoc-adi-9}
\end{eqnarray}

By now the three  linear systems  defined in Eqs. (\ref{hoc-adi-7}, \ref{hoc-adi-8}, \ref{hoc-adi-9}) can be efficiently solved  using Thomas algorithm.
Here  some one-sided fourth-order approximations are needed  for boundary condition approximations in these equation systems.  For example,  in Eq. (\ref{hoc-adi-8c}), the following fourth-order one-sided approximations
 are used to approximate $u_{i,j,1}^*$ and $u_{i, j, N_z}^*$, respectively:
\begin{eqnarray}
& & \hspace{-0.5in} u_{i,j,1}^* = 4 u_{i,j,2}^* -6 u_{i,j,3}^* + 4 u_{i,j,4}^*4 - u_{i,j,5}^*, \nonumber \\ 
& & \hspace{-0.5in}   u_{i,j, N_y}^* =  4 u_{i, j,N_z-1}^* -6 u_{i, j,N_z-2}^* +4 u_{i, j,N_z-3}^* - u_{i, j,N_z-4}^*, \nonumber 
\end{eqnarray}
for $i =2,3,\cdots, N_x-1, \ j =2,3, \cdots, N_y-1$.

The boundary conditions for Eq. (\ref{hoc-adi-8}) can be obtained by setting $j= 1$ and $j = N_y$
 in Eq. (\ref{hoc-adi-8c}), respectively.
\begin{eqnarray}
& & \left(1+\frac{{\delta}_z^2}{12}\right) \frac{u_{i,1,k}^*}{c_{i,1,k}} = \left[ \left(1+\frac{{\delta}_z^2}{12}\right) \frac{1}{c_{i,1,k}} - \frac{{\lambda}_z }{12} \ {\delta}_z^2 \right] {\delta}_t^2 u_{i,1,k}^n, \label{hoc-adi-10} \\
& &  \left(1+\frac{{\delta}_z^2}{12}\right) \frac{u_{i,N_y,k}^*}{c_{i,N_y,k}} = \left[ \left(1+\frac{{\delta}_z^2}{12}\right) \frac{1}{c_{i,N_y,k}} - \frac{{\lambda}_z }{12} \ {\delta}_z^2 \right] {\delta}_t^2 u_{i,N_y,k}^n. \label{hoc-adi-11}
\end{eqnarray}
Solving the two tridiagonal linear systems we can get the boundary conditions for Eq. (\ref{hoc-adi-8}).
Similarly, the boundary conditions needed by Eq. (\ref{hoc-adi-7}) can be obtained by letting $i = 1$ and $i = N_x$, 
respectively.

The new compact ADI method defined in  Eq. (\ref{hoc-adi-9}) is a three-level FD
scheme. Therefore, two initial conditions are needed at $t =0$ and $t = \tau$. 
To approximate the  initial condition at $t = \tau$ with fourth-order accuracy, we expand $u(x_i,y_j,z_k, t)$
 by the Taylor series at $t=0$ and   obtain the following fourth-order approximation 
\begin{equation}
u_{i,j,k}^1 =     u_{i,j,k}^0 + \tau \frac{\partial u}{\partial t}|_{i,j,k}^0  + \frac{{\tau}^2}{2}\frac{{\partial}^2 u}{\partial t^2}|_{i,j,k}^0  +
\frac{{\tau}^3}{6}\frac{{\partial}^3 u}{\partial t^3}|_{i,j,k}^0 + \frac{{\tau}^4}{24}\frac{{\partial}^4 u}{\partial t^4}|_{i,j,k}^0 + O({\tau}^5),
\end{equation}
where the high-order derivatives  are derived using the method in \cite{Liao2018}.

Now we state and prove the main result on the convergence of the compact ADI FD scheme defined in
Eq. (\ref{hoc-adi-4-b}).

\begin{theorem}
Assume that  $u(x,y,z,t) \in C_{x,y,z,t}^{6,6,6,6}(\Omega \times [0,T])$ is the solution of the acoustic wave equation defined in
Eqs.  ( \ref{exact} - \ref{bdry-exact}), and
the coefficient function  satisfies the smooth condition $c(x,y,z) \in C_{x,y,z}^{2,2,2}(\Omega)$.  Then  the 
compact ADI FD scheme defined in Eq. (\ref{hoc-adi-4-b})  is fourth-order accurate in time and space with the  truncation error  $ O({\tau}^4 + h_x^4 + h_y^4+h_z^4)$.
\end{theorem}
\begin{proof}
According to Theorem 3.1, if $u(x,y,,z,t)$ and $c(x,y,z)$ are sufficiently smooth,  the difference 
between the numerical scheme defined in Eq. (\ref{hoc-adi-4-b}) and the numerical scheme defined in
 Eq. (\ref{hoc-adi-new-2}) is $O({\tau}^6) + O(h_x^6) + O(h_y^6) + O(h_z^6)$.

On the other hand, it is known\cite{Liao2014} that  the compact Pad\'{e} approximation FD method 
defined in Eq. (\ref{hoc-adi-4-b}) is fourth-order in time and space, with the truncation error  $O({\tau}^4) + O(h_x^4) + O(h_y^4) +O(h_z^4)$.

Moreover, one can see that the truncation errors caused by the substitutions
\[
\frac{{\delta}_x^2}{(1+ \frac{{\delta}_x^2}{12})} \rightarrow {\delta}_x^2\left(1- \frac{{\delta}_x^2}{12}\right),  
\frac{{\delta}_y^2}{(1+ \frac{{\delta}_y^2}{12})} \rightarrow {\delta}_y^2\left(1- \frac{{\delta}_y^2}{12}\right), 
\frac{{\delta}_z^2}{(1+ \frac{{\delta}_z^2}{12})} \rightarrow {\delta}_z^2\left(1- \frac{{\delta}_z^2}{12}\right)
\]
 in Eq. (\ref{hoc-adi-7}) are $O(h_x^6)$, $O(h_y^6)$ and $O(h_z^6)$, respectively.
Base on these, the new compact ADI FD scheme is fourth-order in time and space.
\qed
\end{proof}

\section{Stability Analysis}
It is important  that a numerical method is stable when it is applied to solve time-dependent problem.
Most of the FD schemes for solving  the acoustic wave equation are conditionally stable and
 subject to constraints on time step. The popular Von Neumann analysis is applicable for constant velocity case only, 
therefore,  in this paper we adopted the energy method in \cite{Britt2018} to analyze and prove  the stability of the new method.  

For the sake of simplicity, assume zero source for the wave equation. Consider the  Pad\'{e} approximation 
based fourth-order finite difference scheme 
\begin{equation}\label{stab-eq-0}
	\frac{\delta_t^2}{1+\frac{1}{12}\delta_t^2}u^n_{i,j,k} = c_{i,j,k} \left[
	\frac{\lambda_x \delta_x^2}{1+\frac{1}{12}\delta_x^2} + \frac{\lambda_y \delta_y^2}{1+\frac{1}{12}\delta_y^2} + \frac{\lambda_z \delta_z^2}{1+\frac{1}{12}\delta_z^2}
	\right]u^n_{i,j,k},
\end{equation} 
where $\lambda_x = \frac{\tau^2}{h_x^2}$, $\lambda_y = \frac{\tau^2}{h_y^2}$, $\lambda_z = \frac{\tau^2}{h_z^2}$. For simplicity we assume $h = h_x = h_y = h_z$, $\lambda = \lambda_x = \lambda_y = \lambda_z$, $\omega = \lambda c_{i,j,k} = \frac{{\tau}^2 c_{i,j,k}}{h^2}$. Note that $\omega$ is a grid function but independent of time. Then the above scheme becomes
\begin{equation}
	\frac{1}{\omega}\frac{\delta_t^2}{1+\frac{1}{12}\delta_t^2}u^n_{i,j,k} = \left [
	\frac{ \delta_x^2}{1+\frac{1}{12}\delta_x^2} + \frac{ \delta_y^2}{1+\frac{1}{12}\delta_y^2} + \frac{ \delta_z^2}{1+\frac{1}{12}\delta_z^2}
	\right]u^n_{i,j,k}.
\end{equation}
If we let
\begin{equation*}
	\begin{split}
		\mathcal{L}  = &
		\frac{ \delta_x^2}{1+\frac{1}{12}\delta_x^2} + \frac{ \delta_y^2}{1+\frac{1}{12}\delta_y^2} + \frac{ \delta_z^2}{1+\frac{1}{12}\delta_z^2}\\
		= & \ T_x + T_y + T_z,
	\end{split}
\end{equation*}
The scheme can be written as
\begin{equation}\label{L-scheme}
	\frac{1}{\omega}\frac{\delta_t^2}{1+\frac{1}{12}\delta_t^2}u^n = \mathcal{L} \ u^n,
\end{equation}
where $u^n $ is  the numerical solution at time level $t_n$:
\[
u^n = (u_{i,j,k}^n)_{N\times N \times N}.
\]
Here we assume that $N_x = N_y = N_z = N$.

To prove the stability result, we first state the following lemma, which  can be found in standard functional analysis textbook such as \cite{Teschl2014}.
\begin{lemma}\label{spectral}
	Let $f(s)$ be a real-valued measurable  function, $A$ be a self-adjoint operator. Then
	\begin{equation*}
		\sigma(f(A)) \subseteq \overline{f(\sigma(A))},
	\end{equation*}
	where $\sigma(A)$ is the spectrum of $A$,  $\overline{f(\sigma(A))}$ the closure of the set $f(\sigma(A))$. In addition, if $f$ is continuous and $\sigma(f(A)) = f(\sigma(A))$, then 
	\[
	\sigma(f(A)) = f(\sigma(A)).
	\]
\end{lemma}

\begin{lemma}
	If $T_x$, $T_y$ and $T_z$ are self-adjoint operators in $L^2$, so is the sum $\mathcal{L} = T_x + T_y + T_z$.
\end{lemma}
\begin{proof}
	Let
	\begin{equation*}
		f(s) = \frac{s}{1+\frac{1}{12}s},
	\end{equation*}
	then $T_x = f(\delta_x^2)$. Note that we can write 
	\begin{equation}
		f(s) = s \cdot \left(1+\frac{1}{12}s\right)^{-1} = \left(1+\frac{1}{12}s\right)^{-1} \cdot s ,
	\end{equation}
	which means 
	\begin{equation}
		\delta_x^2 \cdot \left(1+\frac{1}{12}\delta_x^2\right)^{-1} = f(\delta_x^2) = \left(1+\frac{1}{12}\delta_x^2\right)^{-1} \cdot \delta_x^2.
	\end{equation}
	Thus, if we can prove that both $\delta_x^2$ and $(1+\frac{1}{12}\delta_x^2)^{-1}$ are self-adjoint, then $T_x = f(\delta_x)$, as a product of two commutative self-adjoint operators, is also self-adjoint. It is clear that both $\delta_x^2$ and $1+\frac{1}{12}\delta_x^2$ are self-adjoint, then $(1+\frac{1}{12}\delta_x^2)^{-1}$, as an inverse of a self-adjoint operator, is self-adjoint. The proofs for $T_y$ and $T_z$ are similar. Finally $\mathcal{L}$ is self-adjoint since it is a sum of three self-adjoint operators.
	\qed
\end{proof}	
	The spectrum of $\delta_x^2$ with the homogeneous Dirichlet boundary condition is given by
	\begin{equation*}
		\sigma(\delta_x^2) = \left\{-4 \sin^2\left(\frac{j \ \pi }{2(N + 1)}
		\right)
		\right\} \subset (-4,a(h)],
	\end{equation*}
	where $N$ is the number of grid points in the $x$-direction, $j = 1,\cdots,N$,  $a(h) = -4 \sin^2\big(\frac{\pi h}{2}
	\big)$. Then lemma \ref{spectral} asserts that the spectrum of $T_x$ is given by
	\begin{equation*}
		\sigma(T_x) = f(\sigma(\delta_x^2)) = \left\{f\left(-4 \sin^2\left(\frac{j \ \pi }{2(N + 1)}
		\right)\right)
		\right\},
	\end{equation*}
	where $f(s) = \frac{s}{1+\frac{1}{12}s}$ and $j=1,\cdots,N$. Since $f(s)$ is increasing when $s>-12$, we have 
	\begin{equation*}
		\sigma(T_x) \subset (-6,f(a(h))].
	\end{equation*}
	We have $f(a(h)) < 0 $ since $-12 < a(h)<0$ when $h$ is small enough.
	Note that the operators $T_x$, $T_y$ and $T_z$ are actually hermitian matrices, 
	and the fact that the operator $\mathcal{L}$ corresponds to a homogeneous Dirichlet problem, then (for instance, see \cite{Knutson2001})
	\begin{equation*}
		\sigma(\mathcal{L}) \subset (-18,3f(a(h))].
	\end{equation*}
	Thus, we obtained the coercive condition, which is a direct result of its spectrum estimate,
	\begin{equation}\label{coercivity}
	m = -3f(a(h)) \leq -\mathcal{L} \leq 18 = M.
	\end{equation} 

\begin{remark}
	Note that $h$ is the grid size and should be small enough, then
	\begin{equation*}
		a(h) = -4 \sin^2\big(\frac{\pi h}{2}
		\big) \approx -\pi^2 h^2
	\end{equation*}
\end{remark}

Now we state the main result on the stability of the new method in the following theorem.
\begin{theorem}
\label{stability-thm}
Assume that the solution of the acoustic wave equation Eq. (\ref{exact}) is sufficiently smooth,  the new scheme given in Eq. (\ref{hoc-adi-4-b})
is  stable if  
\[
\max_{1 \le i,j,k \le N}    \frac{{\nu}_{i,j,k} \ \tau}{h} <  \frac{1}{\sqrt{3}}.
\]
\end{theorem}
 
\begin{proof}
Here we follow
 the strategy of \cite{Britt2018} to prove the stability.  First, let's denote the $L^2$ norm by $\|\cdot \|$, the inner product on $L^2$ by $\langle \cdot,\cdot \rangle$.

From (\ref{L-scheme}), since $\delta_t^2$ commutes with $\mathcal{L}$, we have
\begin{equation}\label{L}
\frac{1}{\omega}\delta_t^2 u^n-\frac{1}{12}\mathcal{L}\delta_t^2 u^n = \mathcal{L}u^n.
\end{equation}
Define $v^n = u^{n}-u^{n-1}$, then $\delta_t^2 u^n = v^{n+1}-v^{n}$, $v^{n+1}+v^{n} = u^{n+1}-u^{n-1}$. Taking
 inner product with $v^{n+1}+v^{n}$ on both sides of Eq. (\ref{L}), noting that $\mathcal{L}$ is self-adjoint, we have
\begin{equation}\label{conservation}
\begin{split}
& \langle \frac{1}{\omega}(v^{n+1}-v^{n}),v^{n+1}+v^{n}\rangle - \frac{1}{12}\langle \mathcal{L}(v^{n+1}-v^{n}),v^{n+1}+v^{n}\rangle \\
= & \langle \mathcal{L}u^n,u^{n+1}-u^{n-1}\rangle \\
\end{split}
\end{equation}
Expanding the right-hand side of Eq. (\ref{conservation}) gives
\begin{equation}\label{conservation conti}
	\begin{split}
		& \langle \mathcal{L}u^n,u^{n+1}-u^{n-1}\rangle \\
		= & \frac{1}{4}\big[\langle \mathcal{L}v^n,v^n \rangle - \langle \mathcal{L}v^{n+1},v^{n+1}\rangle -\langle\mathcal{L}(u^n+u^{n-1}),u^n+u^{n-1}\rangle \\
		+ & \langle\mathcal{L}(u^{n+1}+u^{n}),u^{n+1}+u^{n}\rangle
		\big]
	\end{split}
\end{equation}
Combining (\ref{conservation}) and (\ref{conservation conti}), noting that the cross terms in (\ref{conservation}) eliminate since $\mathcal{L}$ is self-adjoint, we have
\begin{equation}
	\begin{split}
		& \frac{1}{\omega}\langle v^n,v^n \rangle + \frac{1}{6}\langle\mathcal{L} v^{n},v^{n}\rangle - \frac{1}{4}\langle\mathcal{L}(u^{n}+u^{n-1}),u^{n}+u^{n-1}\rangle \\
		= & \frac{1}{\omega}\langle v^{n+1},v^{n+1} \rangle + \frac{1}{6}\langle\mathcal{L} v^{n+1},v^{n+1}\rangle - \frac{1}{4}\langle\mathcal{L}(u^{n+1}+u^{n}),u^{n+1}+u^{n}\rangle
	\end{split}
\end{equation}
If we define
\begin{equation*}
	\begin{split}
		S_{n} =  \frac{1}{\omega}\|v^{n}\|^2 + \frac{1}{6}\langle\mathcal{L} v^{n},v^{n}\rangle - \frac{1}{4}\langle\mathcal{L}(u^{n}+u^{n-1}),u^{n}+u^{n-1}\rangle,
	\end{split}
\end{equation*}
then the above equality is exactly
\begin{equation*}
	S_{n} = S_{n+1}.
\end{equation*}
By the coercivity of $-\mathcal{L}$, $m\leq -\mathcal{L} \leq M$, we have
\begin{equation}
	S_{n} \geq \frac{1}{\omega}\|v^n\|^2 - \frac{M}{6}\|v^n\|^2 + \frac{m}{4}\|u^{n}+u^{n-1}\|^2
\end{equation}
and 
\begin{equation}
	S_{n} \leq \frac{1}{\omega}\|v^n\|^2 - \frac{m}{6}\|v^n\|^2 + \frac{M}{4}\|u^{n}+u^{n-1}\|^2.
\end{equation}
Thus, $S_n$ is equivalent to the energy given by
\begin{equation*}
	\begin{split}
		& \|v^n\|^2 + \|u^n + u^{n-1}\|^2 \\
		= &\|u^n - u^{n-1}\|^2 + \|u^n + u^{n-1}\|^2 \\
		= & 2\|u^n\|^2 + 2\|u^{n-1}\|^2
	\end{split}
\end{equation*}
 if and only if 
\begin{equation*}
	\frac{1}{\omega} > \frac{M}{6},
\end{equation*}
i.e.
\begin{equation}\label{stability-condition}
	\max_{i,j,k} {\nu}_{i,j,k}^2 \cdot \frac{\tau^2 }{h^2}<\frac{6}{M} = \frac{1}{3} \Rightarrow \max_{i,j,k} {\nu}_{i,j,k} \cdot \frac{\tau }{h} < \frac{1}{\sqrt{3}}.
\end{equation}
since $M = 18$ by Eq. (\ref{coercivity}).

Denoted by $e_n$, the error at time $t_n$, the above stability analysis shows that the energy of the error $\|e^n\|^2 + \|e^{n-1}\|^2$ conserves during the solving process, which means the scheme is conditionally stable, as long as the stability condition in Eq. (\ref{stability-condition}) is satisfied.
\qed
\end{proof}

For comparison,  according to \cite{Lines1999}, the stability condition for 3-D problem using the standard second-order  difference scheme  is
\[
 \max_{i,j,k} {\nu}_{i,j,k} \cdot  \frac{\tau}{h}  < \frac{1}{3}.
\]
Moreover, the stability condition for the conventional fourth-order FD scheme(it is second-order in time)  is 
\begin{equation}
 \max_{i,j,k} {\nu}_{i,j,k} \cdot  \frac{\tau}{h}  < \frac{1}{2}. 
\end{equation}
Apparently, the new method has better stability with a larger CFL number.
  Although in each time step a sequence of  tridiagonal linear systems need to be solved to march the
numerical solution,  the high-order ADI  method  outperforms  other existing methods in terms of the overall efficiency. One can also 
see that estimated upper bound  of $\frac{\tau}{h}$ is sharp, as demonstrated 
in the second numerical example.

\section{Numerical examples}\label{num}
In this section  three numerical examples  are solved by the new method to demonstrate the efficiency and
accuracy. The exact solutions of the first and second examples  are  available, 
so  the numerical errors can be calculated to validate the order of convergence and stability condition of the new method. 
In the third example,
the acoustic wave equation  with the Ricker's wavelet source is solved to demonstrate that  the new method is
effective in suppressing numerical dispersion and efficient and accurate in simulating wave propagation in heterogeneous media.
It is worthwhile to mention that in the following numerical examples, all numerical errors are calculated using
maximal norm, although the stability and error analysis were conducted  in $L^2$ norm. Since the maximal and $L^2$ norms are
equivalent, all conclusions confirmed in maximal norm hold in $L^2$ norm.
\subsection{ Example 1}

In this example,
the acoustic wave equation is defined on a rectangular domain 
$\Omega = [0,\pi]\times [0,\pi] \times [0,\pi]$, and $t \in [0, T]$,
\begin{eqnarray}
&  & \hspace{-0.3in} u_{tt}  =  \left[1+ \left(\frac{x}{\pi}\right)^2 + \left(\frac{y}{\pi}\right)^2 + \left(\frac{z}{\pi}\right)^2\right] \Delta u + s(x,y,z,t), \ (x,y,z,t) \in \Omega \times [0,T], \nonumber \\
& & \hspace{-0.3in}  u(x,y,z,0)  =   \sin(x) \sin(y) \sin(z), \quad (x,y,z) \in \Omega, \nonumber \\
& & \hspace{-0.3in} u_t(x,y,z,0)  =  0, \quad (x,y,z) \in \Omega, \nonumber \\
& & \hspace{-0.3in} u|_{\partial \Omega}  =  0,  \quad (x,y,z,t) \in \partial \Omega \times [0,T],\nonumber 
\end{eqnarray}
  with the source function given by
\[
s(x,y,z,t) = \left[ 3+ 2\left(\frac{x}{\pi}\right)^2  + 2\left(\frac{y}{\pi}\right)^2  + 2\left(\frac{z}{\pi}\right)^2 \right] \ \cos(t)\sin(x)\sin(y)\sin(z)
\]
and the analytical solution  given by
\[
u(x,y,z,t) = \cos(t)\sin(x)\sin(y)\sin(z).
\]
 It is noted that this problem has zero boundary condition, which is chosen to simplify programming. A more general example
with non-zero boundary condition will be solved in the next example. In all numerical simulations, the domain $\Omega$ is 
divided into an $N_x \times N_y \times N_z$ grid.  The time domain is uniformly divided into $N_t$ subdomains.
To simplify the discussion,  uniform grid size $h$ is used  in $x$, $y$ and $z$ directions. To validate the fourth-order convergence in space,  we 
 fixed $\tau = 0.0025$ so the temporal 
truncation error  is negligible. The  errors in maximal  norm obtained by using different $h$ are included 
in Table \ref{ex1-t1},
which clearly show that  the new method is fourth-order accurate in space. 
Here the numerical order is calculated using the following formula
\[
\text{Conv.  Order} =\frac{ \log(E(h_1)/E(h_2))}{\log(h_1/h_2)},
\]
where the numerical error $E(h)$ is defined by
\[
E(h) = \max_{\substack{1 \le i \le N_x \\  1 \le j \le N_y \\ 1 \le k \le N_z}}  \left| u(x_i,y_j,z_k, T) - u_{i,j,k}^{N_t}  \right|.
\]
Here $u(x_i,y_j,z_k,T)$ is the exact solution of $u$ at the grid point $(x_i, y_j, z_k)$ and time $t = T$, and $u_{i,j,k}^{N_t}$ is the numerical solution at the same grid point and time $t = T$ that is computed using stepsize $h$. The numerical error $E(h,\tau)$ included in Tables (\ref{ex1-t3} - {\ref{ex2-t3}}) is defined similarly,  with the only difference that the time stepsize $\tau$ is changed when $h$ is changed.

It is clear that the errors are reduced roughly by a factor  $16$
when  $h$ is reduced by a factor  $2$. 
We notice that the convergence order is slightly lower than
fourth-order, due to the round-off errors. 
\begin{table}[ht]
\begin{center}
\caption{Numerical errors in maximal norm for {\bf example 1} with $\tau = 0.0025$ at $T = 1$.}
\label{ex1-t1}
\begin{tabular}{|c|c|c|c|c|c|}  \hline
$h$ & $\pi/10$ & $\pi/16$ & $\pi/20$ & $\pi/32$ & $\pi/40$ \\ \hline
$E(h)$ &  4.3196e-04 &  5.4662e-05  & 2.1191e-05 & 2.5748e-06 & 9.9448e-07 \\ \hline
Conv. Order & - & 3.3374 & 3.6593  & 3.8132  & 3.7422   \\ \hline
CPU time(s)  & 5.4699 & 9.1399 &   15.1600 & 42.3399 &  68.5999\\ \hline
\end{tabular}
\end{center}
\end{table}
To show that the method is fourth-order accurate in time, $h$ and $\tau$ are 
simultaneously reduced by the same factor  to ensure that 
the CFL condition  is satisfied, since
using very small  $h$ to verify the order in time will 
  violate the stability condition. Therefore, we verify the order 
of convergence in time through the  following argument.  Suppose the numerical
scheme is  $pth$-order accurate  in time and fourth-order in space, with $p < 4$,  halving  $h$ and $\tau$ several times, the truncation error in time 
will become the dominating error, thus the total error will be reduced by a factor
of $2^p < 16$ when $h$ and $\tau$ been halved. In the following numerical test cases, we start from 
$h=\pi/10, \tau = 1/16$ (the parameters are chosen  to satisfy the  stability condition) and each time we halve 
both $h$ and $\tau$. The result in Table \ref{ex1-t3} clearly indicates that the total error is reduced by 
a factor $16$ (roughly)  when $h$ and $\tau$ are halved, which confirmed that the convergence order in time is fourth-order.
\begin{table}[ht]
\begin{center}
\caption{ Numerical errors in xaximal norm for {\bf example 1} with various $h$ and $\tau$. }
\label{ex1-t3}
\begin{tabular}{|c|c|c|c|c|}  \hline
$(h,  \tau)$  &$ (\pi/10,1/16)$ & $(\pi/20,1/32)$ & $(\pi/40, 1/64)$ & $(\pi/80,1/128)$ \\ \hline
$E(h,\tau)$ & 4.0768e-04  & 2.0790e-05 &  9.5059e-07 & 4.2791e-08  \\ \hline  
$\frac{E(h,\tau)}{E(h/2,\tau/2)}$ & - & 19.6094  & 21.8706  & 22.2147  \\ \hline
Conv. Order & - &   4.2934 &  4.45092 & 4.4734  \\ \hline
\end{tabular}
\end{center}
\end{table}
Furthermore,  it is interesting that the convergence order is slightly higher than 4. One possible
explanation is that the truncation errors are canceled during the computation.
The new method is an implicit scheme, so the computational cost  in each time step
is higher than that of the explicit method,  however, the overall computational efficiency has been greatly
 improved due to the high-order convergence and larger time step size $\tau$ being used.

\subsection{Example 2}

In the second example, we solve a more general acoustic wave equation defined on $[0, \ \pi]\times [0, \ \pi] \times [0, \ \pi] \times [0, \ T]$ with non-zero boundary conditions. The analytical solution for the following
equation 
\[
u_{tt}  = \left(1+ \sin^2x + \sin^2y + \sin^2z\right) (u_{xx} + u_{yy} + u_{zz}) + s(x,y,z,t)
\]
is given by
\[
u(x,y,z,t) = e^{-t}\cos(x)\cos(y)\cos(z),
\]
where $s(x,y,z,t) = (4+ 3(\sin^2x + \sin^2y + \sin^2z))e^{-t}\cos(x)\cos(y)\cos(z)$. 
For this example, the boundary conditions are non-zero. For example, on the plane $x = 0$, the boundary condition is given by
\[
u(0,y,z,t) = e^{-t}\cos(y)\cos(z), \ \   (y, z) \in [0, \ \pi]\times [0, \ \pi], \  t > 0.
\]
To demonstrate the fourth-order convergence in space and time, we reduce $\tau$ and $h$ simultaneously by the same factor and
record the numerical errors in Table \ref{ex2-t1}.
\begin{table}[ht]
\begin{center}
\caption{ Numerical errors in maximal norm for {\bf example 2} with various $h$ and $\tau$. }
\label{ex2-t1}
\begin{tabular}{|c|c|c|c|c|}  \hline
$(h,  \tau)$  &$ (\pi/16,1/20)$ & $(\pi/32,1/40)$ & $(\pi/64, 1/80)$ & $(\pi/128,1/160)$ \\ \hline
$E(h,\tau)$ & 5.1391e-05  & 4.2849e-06 &  3.9569e-07 &  2.9088e-08 \\ \hline  
$\frac{E(h,\tau)}{E(h/2,\tau/2)}$ & - & 11.9935  &  10.8289 &  13.6032\\ \hline
Conv. Order & - &   3.5842 &  3.4368 &  3.7659 \\ \hline
CPU time(s)  & 6.83 & 22.09 &   128.04 & 806.29 \\ \hline
\end{tabular}
\end{center}
\end{table}

We then show that the estimated CFL constraint is sharp. According to the proof, the method is conditionally stable with
the CFL condition $\max_{i,j,k} \ {\nu}_{i,j,k}\frac{\tau}{h} < \frac{1}{\sqrt{3}}$. The maximum value of $\nu$ is given by
\[
{\nu}_{max} =  \sqrt{\max_{0 \le x,y,z \le \pi} ( 1+ \sin^2x + \sin^2y + \sin^2z )} = 2.
\]
Therefore, the following stability condition 
\[
\frac{\tau}{h} < \frac{1}{2\sqrt{3}} \approx 0.288675134594813
\]
is required.

First we choose $\tau$ and $h$ such that $\frac{\tau}{h}$  is slightly less than $\frac{1}{2\sqrt{3}}$ so the stability condition 
given by Theorem \ref{stability-thm} is satisfied. For all test  cases in Table \ref{ex2-t2},  we have
\[
\frac{\tau}{h} = \frac{9}{10 \pi} \approx 0.286478897565412 < \frac{1}{2\sqrt{3}} \approx 0.288675134594813.
\]
The numerical
results in Table \ref{ex2-t2}  confirmed that the numerical method is stable when the 
stability condition is met. As can be seen, when $\tau$ and $h$ are reduced, the numerical error is also reduced, showing a convergence
order between $3.25$ and $4$.  The  noticeable deviation from a perfect fourth oder in convergence is possibly caused by the fact that 
$\frac{\tau}{h}$ is very close to the CFL condition. Thus, the CFL condition might be slightly violated due to some random roundoff error, which then  deteriorates the convergence order.  Nevertheless, the numerical results in Table \ref{ex2-t2} clearly show that the method is stable when $\frac{\tau}{h} < \frac{1}{2\sqrt{3}}$. 
\begin{table}[ht]
\begin{center}
\caption{ Numerical errors in maximal norm for {\bf example 2} with  $\frac{\tau}{h} < \frac{1}{2\sqrt{3}}$. }
\label{ex2-t2}
\begin{tabular}{|c|c|c|c|c|}  \hline
$(h,  \tau)$  &$ (\pi/18,1/20)$ & $(\pi/36,1/40)$ & $(\pi/54,1/60)$ & $(\pi/72, 1/80)$  \\ \hline
$E(h,\tau)$ & 3.3689e-05 & 2.7867e-06 &  6.9049e-07 & 2.7001e-07   \\ \hline  
$\frac{E(h_1,{\tau}_1)}{E(h_2,{\tau}_2)}$ & - & 11.9935  &  4.0358 &   2.5573  \\ \hline
Conv. Order & - &   3.5842 &  3.4410 &  3.2638 \\ \hline
\end{tabular}
\end{center}
\end{table}

We then numerically validate  that the estimated CFL condition is a necessary condition for stability. To this end, we choose 
$\tau$ and $h$ such that $\frac{\tau}{h}$ is slightly greater than $\frac{1}{2\sqrt{3}}$, thus, the stability condition given by
Theorem \ref{stability-thm} is slightly violated. As shown in Table \ref{ex2-t3}, 
the ratio $\frac{\tau}{h} = \frac{19}{20 \pi} \approx 0.302394 > \frac{1}{2\sqrt{3}}$.
The numerical results clearly show that the method is not stable, as the numerical solution is not convergent, when 
$\tau$ and $h$ are reduced. Instead, 
when $\tau$ and $h$ are reduced, the maximal error increases  and goes to infinity.
\begin{table}[ht]
\begin{center}
\caption{ Numerical errors in maximal norm for {\bf example 2} with  $\frac{\tau}{h} > \frac{1}{2\sqrt{3}}$. }
\label{ex2-t3}
\begin{tabular}{|c|c|c|c|c|}  \hline
$(h,  \tau)$  &$ (\pi/19,1/20)$ & $(\pi/38,1/40)$ & $(\pi/57,1/60)$ & $(\pi/76, 1/80)$  \\ \hline
$E(h,\tau)$ & 8.5884e-05 & 1.1864e-05 &   7.2660e-01 &  1.3165e+06  \\ \hline  
$\frac{E(h_1,{\tau}_1)}{E(h_2,{\tau}_2)}$ & - &  7.2387 &  $\times$  & $\times$  \\ \hline
Conv. Order & - & 2.8557   & $\times$   & $\times$   \\ \hline
\end{tabular}
\end{center}
\end{table}
Here \enquote{$\times$} in the table means that no convergence is obtained.

\subsection{ Example 3}

To demonstrate the efficiency of the new method and show the effectiveness in suppressing numerical dispersion, we solve
a realistic problem in which the seismic wave is generated by  a Ricker wavelet source located 
at the centre of a three-dimensional
domain   $[0m,\ 1280 m]\times [0m,\ 1280 m] \times [0m, \ 800 m]$. The velocity model is given by
\[
 \nu(x,y,z) = 1200 + 400 \left(\frac{x}{x_{max}}\right)^2 + 100 \left(\frac{y}{y_{max}}\right)^2  + 800 \left(\frac{z}{z_{max}}\right)^2,
\]
where $x_{max}= 1280m$, $y_{max}= 1280m$ and $z_{max} = 800m$, respectively.  
Therefore, the maximal  and minimal wave speeds over the domain are
\[
{\nu}_{max} = \max_{ x,y,z } \ \nu(x,y,z) = 2500 m/s
\]
and 
\[
{\nu}_{min} = \min_{x,y,z } \  \nu(x,y,z) = 1200 m/s,
\]
 respectively.
The  Ricker wavelet source function is given by
\[
s(x,y,z, t) = \delta(x-x_0,y-y_0,z-z_0)\left[1-2{\pi}^2 f_p^2(t-d_r)^2\right]e^{-{\pi}^2 f_p^2(t-d_r)^2},
\]
 where $f_{p}=10 Hz$ is the dominant frequency, $d_r = 0.5/f_p$ is the temporal delay to ensure zero initial conditions.  
The centre of the domain is located at $(x_0,y_0,z_0) = (640, 640m, 400m)$. The space and time step sizes are chosen to
satisfy the CFL condition. 
Moreover, the Nyquist sampling theorem states that {\bf the sampling frequency should be at least twice the highest
frequency contained in the signal to avoid aliasing}.  As a rule of thumb, 
at least 10 grid points per wavelength are required in finite difference discretization. Simple calculation shows that the minimal 
wavelength is given by ${\nu}_{min}/f_p = 120 m$, which sets the upper limit for $h$ as  $h_{max} = 12 m$. 
On the other hand, the CFL condition indicates that
\[
{\nu}_{max} \ \frac{\tau}{h} < \frac{1}{\sqrt{3}} \Rightarrow \frac{\tau}{h} < \frac{1}{2500 \sqrt{3}}
\]
 For all numerical simulations of this example, the uniform grid $h =10m$ and $\tau = 0.001 s$ are used to ensure stability and avoid aliasing. 

We plot the  wavefield snapshots  for central slices in $x-, y- $ and $z-$ directions at $t = 0.2s,  t = 0.3s, t = 0.4s $ and $t = 0.5s$ in 
Figs. (\ref{ex3-fig-1} -  \ref{ex3-fig-3}), respectively.  Apparently the computed wave fronts accurately reflect the wave velocity.  For example,
 in Fig. \ref{ex3-fig-1} the simulated wavefields accurately
describe the wave propagation.  The x-sections are the slices of the 3D wavefields at $x = x_{max}/2$. Since the wave
velocity increases more rapidly in $z$ direction than it does in $y$ direction, we clearly see that the wavefront moves faster in $z$ direction.
Moreover,  the wave velocity is a monotone increasing function of $y$ and $z$, the wavefronts moves faster at the right side  than at the left side, and moves faster in the lower side than in the upper side of  the domain.

Fig. \ref{ex3-fig-2} shows the snapshots of the wavefields  when $y= y_{max}/2$ is fixed. Again, the velocity is a function 
of two variables $x$ and $z$, and
the velocity increases faster in $z$. Clearly the wavefronts moves faster in $z$ direction. While in the same direction, the wave propagates 
faster when the variable is at the large side. For example, the wavefront hits the right boundary before it hits the left boundary.

Finally, in Fig. \ref{ex3-fig-3},  the wavefields slices demonstrate similar wave propagations, which is also expected from theoretical analysis.

In terms of the numerical dispersion, significant improvement can be observed as well.  As shown in these figures,  there is no
visible numerical dispersions, which indicated that the new numerical algorithm is accurate and effective in suppressing 
numerical dispersion. 

\begin{figure}[ht]
\begin{center}
 \includegraphics[width=0.99\textwidth]{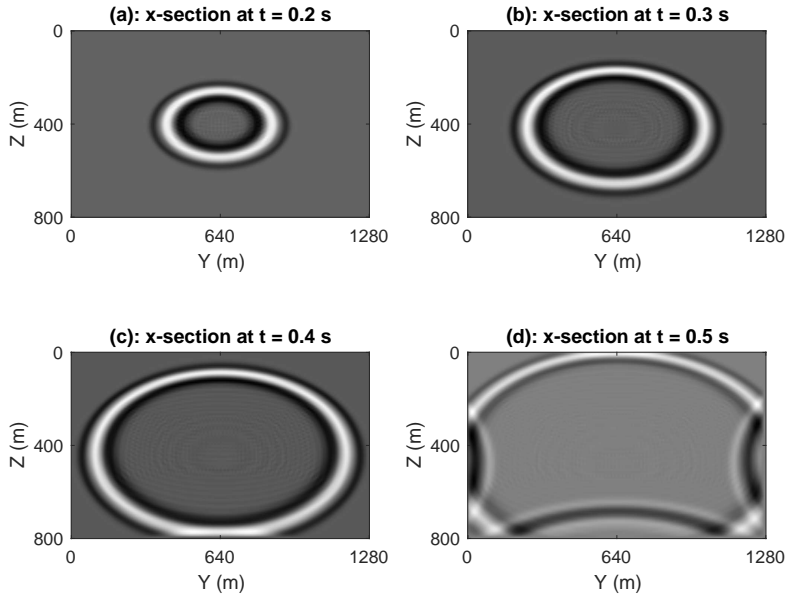}
\caption{Snapshots of x-section of wavefields computed by the new fourth-order compact  method at  (a) t = 0.2 seconds;  (b) t = 0.3 seconds;  (c) 
t = 0.4 seconds; (d) t = 0.5 seconds.}
\label{ex3-fig-1}
\end{center}
\end{figure}

\begin{figure}[ht]
\begin{center}
 \includegraphics[width=0.99\textwidth]{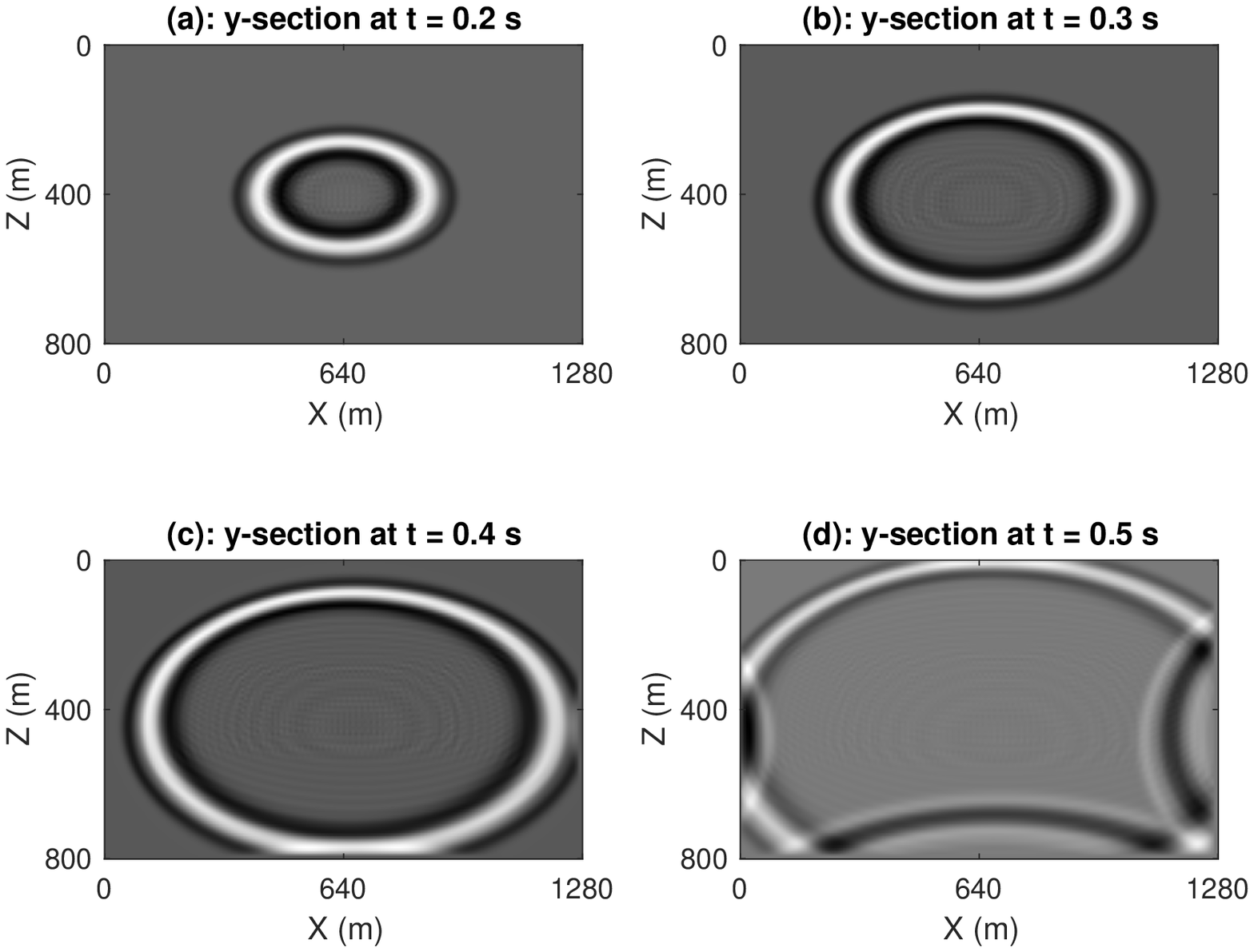}
\caption{Snapshots of y-section of wavefields computed by the new fourth-order compact  method at  (a) t = 0.2 seconds;  (b) t = 0.3 seconds;  (c) 
t = 0.4 seconds; (d) t = 0.5 seconds.}
\label{ex3-fig-2}
\end{center}
\end{figure}

\begin{figure}[ht]
\begin{center}
 \includegraphics[width=0.99\textwidth]{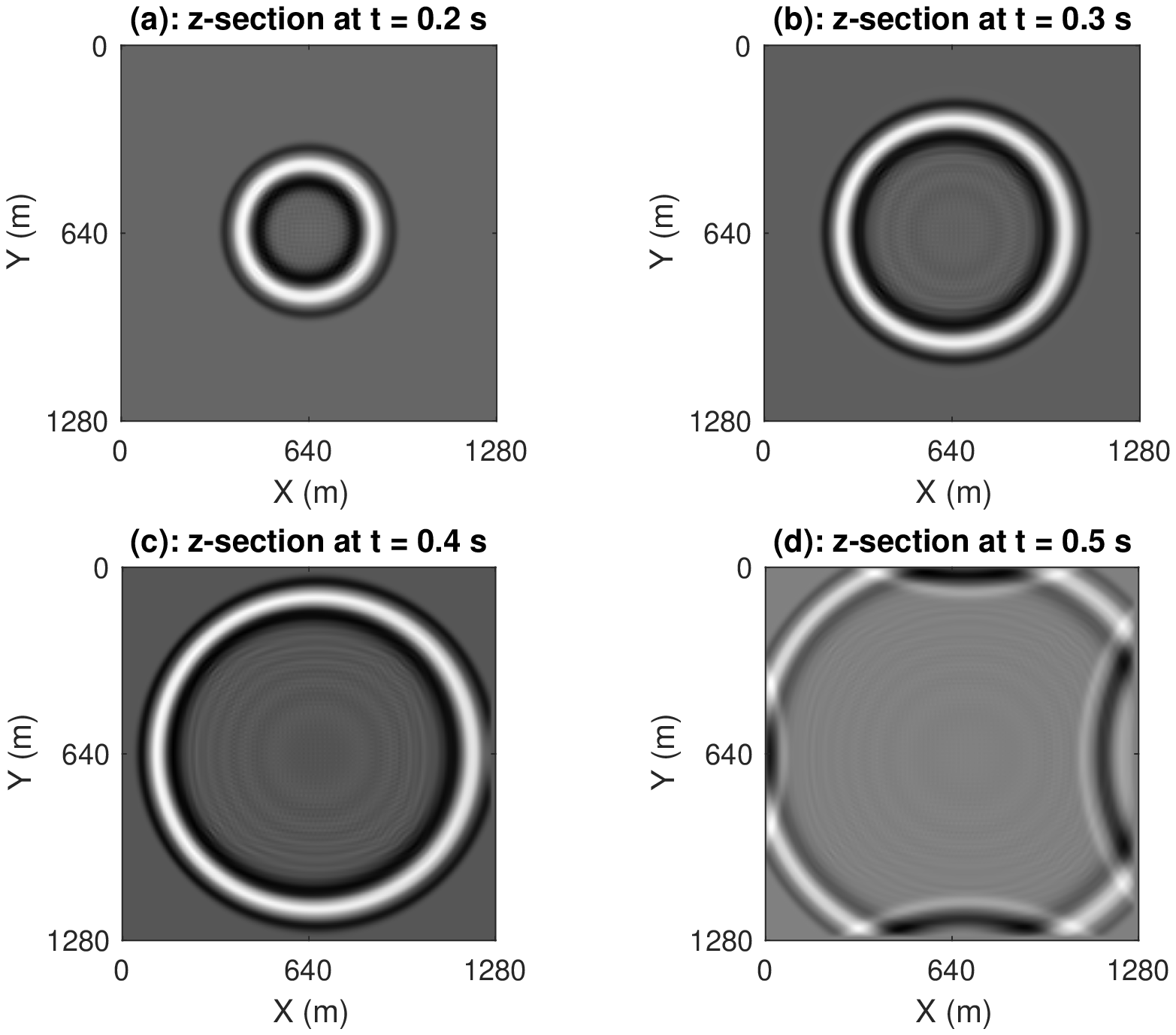}
\caption{Snapshots of z-section of wavefields computed by the new fourth-order compact  method at  (a) t = 0.2 seconds;  (b) t = 0.3 seconds;  (c) 
t = 0.4 seconds; (d) t = 0.5 seconds.}
\label{ex3-fig-3}
\end{center}
\end{figure}

\section{Conclusion and future work}
A compact  fourth-order  ADI FD scheme 
has been developed  to solve the three-dimensional acoustic wave equation in heterogeneous media.
The new method is efficiently implemented using the ADI technique, which splits the original three-dimensional problem into a series of one-dimensional problems. 
The fourth-order convergence  in time and space has been
 validated by two numerical examples for which the exact solutions are available. 
A more realistic problem   has been solved  to  demonstrate that the new method  is robust,  accurate and  efficient
 in seismic wave propagation simulation.  Moreover, the conditional stability of the new method has been rigidly proved for the
variable coefficient case. It has been shown that the new method has a larger CFL number than other conventional finite difference methods. 
Several numerical tests has been performed to verify that the estimated upper bounds of CFL is sharp. It is expected that this new method will find extensive
applications in  numerical seismic  modelling on complex geological models, and seismic inversion problems.
 In the future we plan to take more realistic boundary condition for instance the  absorbing boundary condition, or 
the perfectly matched layer boundary condition into consideration.

\section*{Acknowledgments}
 The research work   is supported by the Canada NSERC  individual Discovery Grant and  the Mitacs Elevate program. We deeply appreciate
the insightful and constructive comments by the anonymous reviewers.

\appendix
\section{Proof of Theorem 3.1}
{\bf Proof:}

First, let's expand ${\delta}_t^2 u_{i,j,k}^n$ by Taylor series at time $t_n$ and the grid point $(x_i, y_j, z_k)$ as
\begin{eqnarray}
\hspace{-1.5in} {\delta}_t^2 u_{i,j,k}^n & = &  u_{i,j,k}^{n+1} -2 u_{i,j,k}^n +u_{i,j,k}^{n-1} \nonumber \\
&   = &  {\tau}^2 \  \frac{{\partial}^2 u}{\partial t^2}|_{i,j,k}^n +\frac{{\tau}^4}{12}  \ \frac{{\partial}^4 u}{\partial t^4}|_{i,j,k}^n +
\frac{{\tau}^6}{360}  \frac{{\partial}^6 u(x_i,y_j,z_k,{\tau}_n^*)}{\partial t^6}  \label{error-term-2},
\end{eqnarray}
where ${\tau}_n^* \in (t_{n-1}, t_{n+1})$.

By  Pad\'{e} approximation,    if $v(x,y,z,t)$ is a sufficiently smooth function, we have
\begin{equation}\label{error-term-3}
 \frac{{\delta}_y^2}{ 1+\frac{1}{12}{\delta}_y^2 } v_{i,j,k}^n = {\delta}_y^2\left(1 -\frac{{\delta}_y^2}{12}\right) v_{i,j,k}^n  + \frac{ h_y^6}{144} \frac{{\partial}^6 v(x_i,{y}_j^*,z_k,t_n)}{\partial y^6},
\end{equation}
where ${y}_j^* \in (y_{j-1},y_{j+1})$.

On the other hand, 
\begin{eqnarray}
{\delta}_y^2 \left(1-\frac{1}{12}{\delta}_y^2\right) v_{i,j,k}^n & = & \left ({\delta}_y^2 - \frac{1}{12}{\delta}_y^2 {\delta}_y^2 \right) v_{i,j,k}^n \nonumber\\
& &\hspace{-1.5in}  = \frac{-v_{i,j+2,k}^n + 16 v_{i,j+1,k}^n  -30 v_{i,j,k}^n  + 16v_{i,j-1,k}^n  - v_{i,j-2,k}^n}{12}  \nonumber \\
& & \hspace{-1.5in} = h_y^2 \  \frac{{\partial}^2 v}{\partial y^2}|_{i,j,k}^n
-\frac{2}{15} h_y^6 \  \frac{{\partial}^6 v(x_i,{y}_j^{**},z_k,t_n)}{\partial y^6},  \label{error-term-5a}
\end{eqnarray}
where $y_j^{**} \in (y_{j-1}, y_{j+1})$.

Letting $v_{i,j,k}^n = {\delta}_t^2 u_{i,j,k}^n$ and combining Eq. (\ref{error-term-5}) with Eq. (\ref{error-term-5a})  lead to
\begin{equation}\label{error-term-5}
\frac{{\delta}_y^2}{ 1+\frac{1}{12}{\delta}_y^2 } \ {\delta}_t^2 u_{i,j,k}^n   =  h_y^2 \  \frac{{\partial}^2 v}{\partial y^2}|_{i,j,k}^n
-\frac{2 h_y^6}{15}  \  \frac{{\partial}^6 v(x_i,y_j^{**},z_k,t_n)}{\partial y^6} + \frac{h_y^6}{144} \  \frac{{\partial}^6 v(x_i,y_j^{*}, z_k,t_n)}{\partial y^6}.
\end{equation}

Using Eq. (\ref{error-term-2}), it then follows that
\begin{eqnarray}
& & {\lambda}_y \  c_{i,j,k}  \frac{{\delta}_y^2}{ 1+\frac{1}{12}{\delta}_y^2 } {\delta}_t^2 u_{i,j,k}^n   \nonumber \\
 & &  = {\lambda}_y \  c_{i,j,k}  \frac{{\delta}_y^2}{ 1+\frac{1}{12}{\delta}_y^2 } \left[{\tau}^2 \  \frac{{\partial}^2 u}{\partial t^2}|_{i,j,k}^n +\frac{{\tau}^4}{12}  \ \frac{{\partial}^4 u}{\partial t^4}|_{i,j,k}^n +
\frac{{\tau}^6}{360}  \frac{{\partial}^6 u(x_i,y_j,z_k,{\tau}_n^*)}{\partial t^6} \right] \nonumber \\
& &  = {\tau}^4 \left [ c_{i,j,k} \frac{{\partial}^4 u}{\partial y^2 \partial t^2}|_{i,j,k}^n \right]  + O( h_y^6) + O({\tau}^6) + O({\tau}^4 h_y^4). \label{error-term-6}
 \end{eqnarray}

Further, let $w(x,y,z,t) = c(x,y,z) \frac{{\partial }^4 u}{\partial y^2 \partial t^2}$,  then the first term on the right-hand side of Eq. (\ref{error-term-1}) can be written as
\begin{eqnarray}
& & \hspace{-0.4in} \frac{{\lambda}_x }{144} \ c_{i,j,k}  \frac{{\delta}_x^2}{ 1+\frac{{\delta}_x^2}{12} } \  {\lambda}_y \  c_{i,j,k} 
 \frac{{\delta}_y^2}{ 1+\frac{{\delta}_y^2}{12} } {\delta}_t^2 u_{i,j,k}^n  \nonumber \\
& &  \hspace{-0.4in}  = {\tau}^4\ \frac{c_{i,j,k}}{144}  \ \left[ \frac{{\tau}^2}{h_x^2} \  \frac{{\delta}_x^2}{ 1+\frac{1}{12}{\delta}_x^2} \  w_{i,j,k}^n \right]   + O(h_y^6)+O({\tau}^6) + O({\tau}^4 h_y^4) \label{error-term-7}\\
& & \hspace{-0.4in} =  {\tau}^4\ \frac{c_{i,j,k}}{144}  \ \left[  \frac{{\tau}^2}{h_x^2} \ \left( {\delta}_x^2 \left(1-\frac{1}{12}{\delta}_x^2\right)  \  w_{i,j,k}^n  +O(h_x^6) \right) \right] +  O(h_y^6)+O({\tau}^6)   + O({\tau}^4 h_y^4). \nonumber
\end{eqnarray}
Expanding ${\delta}_x^2 (1-\frac{1}{12}{\delta}_x^2)  \  w_{i,j,k}^n$ we obtain
\begin{eqnarray}
& & {\delta}_x^2 \left(1-\frac{1}{12}{\delta}_x^2\right)  \  w_{i,j,k}^n =\nonumber \\
& &  \frac{1}{12}\left[-w_{i-2,j,k}^n +16w_{i-1,j,k}^n -30w_{i,j,k}^n +16w_{i+1,j,k}^n - w_{i+2,j,k}^n\right]. \label{error-term-8}
\end{eqnarray}
 Using  Taylor series expansion, we can simplify it to
\begin{equation}\label{error-term-9}
{\delta}_x^2 (1-\frac{1}{12}{\delta}_x^2)  \  w_{i,j,k}^n = h_x^2 \frac{{\partial}^2 w}{\partial x^2}|_{i,j,k}^n - \frac{2 h_x^6}{15} \ \frac{{\partial}^6 w(x_i^*,y_j,z_k,t_n)}{\partial x^6},
\end{equation}
where $x_i^* \in (x_{i-1}, x_{i+1})$.

Inserting Eq. (\ref{error-term-9}) into Eq. (\ref{error-term-7}) leads to
\begin{eqnarray}
& &  \frac{{\lambda}_x }{144} \ c_{i,j,k}  \frac{{\delta}_x^2}{ 1+\frac{{\delta}_x^2}{12} } \  {\lambda}_y \  c_{i,j,k} 
 \frac{{\delta}_y^2}{ 1+\frac{{\delta}_y^2}{12} } {\delta}_t^2 u_{i,j,k}^n \nonumber \\
& &  =  {\tau}^4\ \frac{c_{i,j,k}}{144}  \ \left[  {\tau}^2 \ \frac{{\partial}^2 w}{\partial x^2}|_{i,j,k}^n \right]    + O({\tau}^6) + O(h_y^6) +O({\tau}^6 h_y^4)  \nonumber \\
& & =  {\tau}^6 \ \frac{c_{i,j,k}}{144}  \    \frac{{\partial}^2 w}{\partial x^2}|_{i,j,k}^n  + O({\tau}^6) 
+   O(h_y^6),\label{error-term-10}
\end{eqnarray}
where
\begin{equation}\label{error-term-11}
\frac{{\partial}^2 w}{\partial x^2} = \frac{{\partial}^2 c}{\partial x^2}  \ \frac{{\partial }^4 u}{\partial y^2 \partial t^2}
+2 \  \frac{\partial c}{\partial x} \ \frac{{\partial }^5 u}{\partial x \partial y^2 \partial t^2} + c(x,y,z) \frac{{\partial }^6 u}{\partial x^2 \partial y^2 \partial t^2}.
\end{equation}

Similarly, we can derive the error estimations of other terms in Eq. (\ref{error-term-1}) as the follows:
\begin{eqnarray}
& &  \frac{{\lambda}_y }{144} \ c_{i,j,k}  \frac{{\delta}_y^2}{ 1+\frac{{\delta}_y^2}{12} } \  {\lambda}_z \  c_{i,j,k} 
 \frac{{\delta}_z^2}{ 1+\frac{{\delta}_z^2}{12} } {\delta}_t^2 u_{i,j,k}^n \nonumber \\
& &  =  {\tau}^4\ \frac{c_{i,j,k}}{144}  \ \left[  {\tau}^2 \ \frac{{\partial}^2 \bar{w}}{\partial y^2}|_{i,j,k}^n \right]    + O({\tau}^6) + O(h_z^6)  \nonumber \\
& & =  {\tau}^6 \ \frac{c_{i,j,k}}{144}  \    \frac{{\partial}^2 \bar{w}}{\partial y^2}|_{i,j,k}^n  + O({\tau}^6) 
+   O(h_z^6),\label{error-term-10-b}
\end{eqnarray}
where 
\begin{equation}\label{error-term-11-b}
\frac{{\partial}^2 \bar{w}}{\partial y^2} = \frac{{\partial}^2 c}{\partial y^2}  \ \frac{{\partial }^4 u}{\partial z^2 \partial t^2}
+2 \  \frac{\partial c}{\partial y} \ \frac{{\partial }^5 u}{\partial y \partial z^2 \partial t^2} + c(x,y,z) \frac{{\partial }^6 u}{\partial y^2 \partial z^2 \partial t^2},
\end{equation}
\begin{eqnarray}
& &  \frac{{\lambda}_x }{144} \ c_{i,j,k}  \frac{{\delta}_x^2}{ 1+\frac{{\delta}_x^2}{12} } \  {\lambda}_z \  c_{i,j,k} 
 \frac{{\delta}_z^2}{ 1+\frac{{\delta}_z^2}{12} } {\delta}_t^2 u_{i,j,k}^n   \nonumber \\
& &  =  {\tau}^4\ \frac{c_{i,j,k}}{144}  \ \left[  {\tau}^2 \ \frac{{\partial}^2 \tilde{w}}{\partial x^2}|_{i,j,k}^n \right]    + O({\tau}^6) + O(h_z^6)  \nonumber \\
& & =  {\tau}^6 \ \frac{c_{i,j,k}}{144}  \    \frac{{\partial}^2 \tilde{w}}{\partial x^2}|_{i,j,k}^n  + O({\tau}^6) 
+   O(h_z^6),\label{error-term-10-c}
\end{eqnarray}
where
\begin{equation}\label{error-term-11-c}
\frac{{\partial}^2 \tilde{w}}{\partial x^2} = \frac{{\partial}^2 c}{\partial x^2}  \ \frac{{\partial }^4 u}{\partial z^2 \partial t^2}
+2 \  \frac{\partial c}{\partial x} \ \frac{{\partial }^5 u}{\partial x \partial z^2 \partial t^2} + c(x,y,z) \frac{{\partial }^6 u}{\partial x^2 \partial z^2 \partial t^2},
\end{equation}
and
\begin{eqnarray}
& &  -\frac{{\lambda}_x }{1728} \ c_{i,j,k}  \frac{{\delta}_x^2}{ 1+\frac{{\delta}_x^2}{12} } \  {\lambda}_y \  c_{i,j,k} 
 \frac{{\delta}_y^2}{ 1+\frac{{\delta}_y^2}{12} }  \ {\lambda}_z \  c_{i,j,k} 
 \frac{{\delta}_z^2}{ 1+\frac{{\delta}_z^2}{12} } {\delta}_t^2 u_{i,j,k}^n \nonumber \\
& &  =   -\frac{{\tau}^8 }{1728} \  c_{i,j,k}  \ \left[   \frac{{\partial}^2 }{\partial x^2} \left(\frac{{\partial}^2 \bar{w}}{\partial y^2}
\right)\right]_{i,j,k}^n     + O({\tau}^6) + O(h_y^6),\label{error-term-11-d}
\end{eqnarray}
where $\frac{{\partial}^2 \bar{w}}{\partial y^2} $ is defined in Eq. (\ref{error-term-11-b}).
Therefore, the factoring error $ERR$ in Eq. (\ref{error-term-1}) is given by
\begin{equation}\label{error-term-12}
ERR   =  \bar{M}_t \ {\tau}^6 + \bar{M}_x \ h_x^6 + + \bar{M}_y \ h_y^6  ++ \bar{M}_z \ h_z^6 ,
\end{equation}
provided that the following functions
\[
c(x,y,z), \ \frac{\partial c(x,y,z) }{\partial x^{m_1} \partial y^{m_2} \partial z^{m_3}}, \  
\frac{{\partial}^2 c(x,y,z)}{\partial x^{n_1}\partial y^{n_2}\partial z^{n_3}}
\]
are bounded in $\Omega$, where the non-negative integers satisfying  $m_1 + m_2 + m_3 = 1$ and 
$n_1 + n_2 + n_3 = 2$. Moreover, the solution $u(x,y,z,t)$ and its' derivatives $
 \frac{{\partial }^6 u(x,y,z,t)}{\partial x^{k_1} \partial y^{k_2} \partial z^{k_3} \partial t^{k_4}}$
are bounded in $\Omega \times [0,T]$, where the non-negative integers satisfy $k_1 + k_2 + k_3 + k_4 = 6$.
$\bar{M}_t$, $\bar{M}_x$, $\bar{M}_y$ and $\bar{M}_z$ are positive constants depending on the functions listed above.


\begin{thebibliography}{99}
\bibitem{Bayliss1986} A. Bayliss, K.E. Jordan, B. Lemesurier, E. Turkel,  A fourth-order accurate finite difference scheme
for the computation of elastic waves, Bull. Seismol. Soc. Amer. 76(4)(1986)1115--1132.

\bibitem{Britt2018}
S. Britt, E. Turkel, S. Tsynkov, A High Order Compact Time/Space Finite Difference Scheme for the Wave Equation with Variable Speed of Sound,
 Journal of Scientific Computing (2018): 1-35.

\bibitem{Chen2007} J. B. Chen, High-order time discretizations in seismic modelling, Geophysics, 72 (2007)115--122.

\bibitem{Chu1998} P. Chu, C. Fan, A three-point combined compact difference scheme,
Journal of Computational Physics, 140(1998)370--399.

\bibitem{Etgen2007}
Etgen, John T., Michael J. O'Brien. Computational methods for large-scale 3D acoustic finite-difference modeling: A tutorial.
Geophysics, 72.5 (2007): SM223-SM230.

\bibitem{Cohen1996}
G. Cohen, P. Joly, Construction and analysis of fourth-order finite difference schemes for the acoustic wave equation in non-homogeneous media, 
SIAM J. Numer. Anal., 4(1996)1266--1302.

\bibitem{Dablain1986} M. A. Dablain, The application of high order differencing for the scalar wave equation, Geophysics, 51(1)(1986)54--66.

\bibitem{Das2014} 
S. Das, W. Liao, A. Gupta,  An efficient fourth-order low dispersive finite difference scheme for a 2-D acoustic wave equation.
Journal of computational and Applied Mathematics, 258 (2014): 151-167.

\bibitem{Douglas1966}
J. Douglas Jr., J. Gunn, A general formulation of alternating direction methods part I. Parabolic and hyperbolic problems, Numer. Math. 6 (1966)428--453.

\bibitem{Fairweather1965}
G. Fairweather, A.R. Mitchell, A high accuracy alternating direction method for the wave equation, J. Inst. Math. Appl. 1(1965) 309--316.

\bibitem{Finkelstein2007} B. Finkelstein,  R. Kastner, Finite difference time domain dispersion reduction schemes, J. Comput. Phys.  221 (2007) 422--438.
 
\bibitem{Kelly1976}
K.R. Kelly, R.W. Ward, S. Treitel, E.M. Alford, Synthetic seismograms: A finite difference approach, Geophysics 41(1976)2--27.

\bibitem{Kim2007} 
S. Kim, H. Lim,  High-order schemes for acoustic waveform simulation, Applied Numerical Mathematics 57(2007)402--414.

\bibitem{Knutson2001}
A. Knutson, T. Tao,  Honeycombs and sums of Hermitian matrices, Notices Amer. Math. Soc 48.2 (2001).

\bibitem{Lees1962}
M. Lees, Alternating direction methods for hyperbolic differentials equations, Journal of the society for industrial and applied mathematics 10(1962)610--616.

\bibitem{Levander1988}  A. R.  Levander, Fourth-order finite-difference P-SV seismograms,
Geophysics 53(11)(1988)1425--1436. 

\bibitem{Li2006}
J. Li,  M.R. Visbal, High-order compact schemes for nonlinear dispersive waves, Journal of Scientific Computing 26(1) (2006), 1--23.

\bibitem{Li1991} Z. Li, Compensating finite-difference errors in 3-D migration and modelling, Geophysics 56(1991)1650--1660.

\bibitem{Ristow1997} D. Ristow, T. Ruhl, 3-D implicit finite-difference migration by multiway splitting, Geophysics 62(1997)554--567.

\bibitem{Liao2014} W. Liao,  On the dispersion, stability and accuracy of a compact higher-order finite difference scheme for 3D acoustic wave equation,
Journal of Computational and Applied Mathematics 270 (2014): 571-583.

\bibitem{Liao2018}
W. Liao, P. Yong, H. Dastour, J. Huang, Efficient and accurate numerical simulation of acoustic wave propagation in a 2D heterogeneous media,
Applied Mathematics and Computation 321 (2018): 385-400.
 
\bibitem{Lines1999}
L. Lines, R. Slawinski, R. Bording, A recipe for stability of finite-difference wave-equation computations, Geophysics, 64(3)(1999) 967--969.

\bibitem{Liu2009a}
Y. Liu, M. K. Sen, An implicit staggered-grid finite-difference method for seismic modelling, Geophys. J. Int. 179 (2009) 459--474. 

\bibitem{Nita2008}
B. G. Nita,  Forward scattering series and Pad\'{e} approximants for acoustic wavefield propagation in a vertically 
varying medium, Communications in Computational Physics 3(1)(2008)180--202.

\bibitem{Peaceman1955} G.W. Peaceman, H.H. Rachford, The numerical solution of parabolic and elliptic differential equations,
Journal of the society for industrial and applied mathematics 3(1)(1955)28--41.

\bibitem{Shragge2017}
J. Shragge,  B. Tapley,  Solving the Tensorial 3D Acoustic Wave Equation: A Mimetic Finite-Difference Time-Domain Approach,
Geophysics 82.4 (2017): 1-58.

\bibitem{Shubin1987} G. R. Shubin and J. B. Bell,  A modified equation approach to constructing fourth-order 
methods for acoustic wave propagation, SIAM J. Sci. Statits. Comput  8(2)(1987)135--151.

\bibitem{Shukla2005} R. Shukla, X. Zhong,  Derivation of high-order compact finite difference schemes for 
non-uniform grid using polynomial interpolation, Journal of Computational Physics 204(2)(2005)404--429.

\bibitem{Takeuchi2000} N. Takeuchi, R.J. Geller,  Optimally accurate second order time-domain finite difference scheme for computing synthetic seismograms in  2-D and 3-D media,  Phys. Earth Planet. Int. 119(2000) 99 --131.

\bibitem{Teschl2014}
Teschl, Gerald. Mathematical methods in quantum mechanics. Vol. 157. American Mathematical Soc., 2014.

\bibitem{Yangdh2009} D.H. Yang, N. Wang, S. Chen, G.J. Song,  An explicit method based on the implicit Runge-Kutta algorithm for solving the wave equations, Bulletin of the Seismological Society of America 99(6)(2009) 3340--3354.

\bibitem{Yangdh2010} D.H. Yang,  L. Wang,  A split-step algorithm with effectively suppressing the numerical dispersion for 3D seismic propagation modeling, Bulletin of the Seismological Society of America 100( 4)(2010)1470--1484.

\bibitem{Zhang2011}
Zhang, Wensheng, Li Tong, Eric T. Chung, A new high accuracy locally one-dimensional scheme for the wave equation, Journal of Computational and Applied Mathematics 236.6 (2011) 1343-1353.
\end{thebibliography}
\end{document}